\documentclass[a4paper, twoside, 11pt, english]{article}

\usepackage[T1]{fontenc}
\usepackage[utf8]{inputenc}

\usepackage{etoolbox}

\usepackage{amsmath,amsthm,amssymb,stmaryrd}
\usepackage{mathtools}
\usepackage[ttscale=.875]{libertine}
\usepackage[libertine]{newtxmath}

\usepackage{biblatex}
\usepackage{comment}
\usepackage{ifthen}
\usepackage{algorithm2e}
\usepackage{multicol}

\usepackage{fancyhdr, titlesec, url, enumerate, microtype,setspace}

\usepackage[all,knot,poly]{xy}
\usepackage{tikz}
\usetikzlibrary{decorations.pathreplacing}

\usepackage{tikz-qtree,varwidth}
\usepackage{youngtab}
\usepackage{ytableau}

\usepackage[english]{babel}

\usepackage{hyperref}

\usetikzlibrary{calc}

\newcount\hh
\newcount\mm
\mm=\time
\hh=\time
\divide\hh by 60
\divide\mm by 60
\multiply\mm by 60
\mm=-\mm
\advance\mm by \time
\def\hhmm{\number\hh:\ifnum\mm<10{}0\fi\number\mm}

\newcommand{\periodafter}[1]{\ifstrempty{#1}{}{#1.}}
\titleformat{\section}[block]{\scshape\filcenter\LARGE\boldmath}{\thesection.}{.5em}{}
\titleformat{\subsection}[block]{\bfseries\filcenter\large\boldmath}{\thesubsection.}{.5em}{\medskip}
\titleformat{\subsubsection}[runin]{\bfseries\boldmath}{\thesubsubsection.}{.5em}{\periodafter}
\titlespacing{\subsubsection}{0pt}{\topsep}{.5em}

\newtheoremstyle{ntheorem}%
	{\topsep}{\topsep}{\itshape}{0pt}{\bfseries}{.}{.5em}%
	{\thmnumber{#2.\hspace{.5em}}\thmname{#1}\thmnote{ (#3)}}
	
\newtheoremstyle{ndefinition}%
	{\topsep}{\topsep}{\normalfont}{0pt}{\bfseries}{.}{.5em}%
	{\thmnumber{#2.\hspace{.5em}}\thmname{#1}\thmnote{ (#3)}}
	
\newtheoremstyle{nremark}%
	{\topsep}{\topsep}{\normalfont}{0pt}{\itshape}{.}{.5em}%
	{\thmnumber{}\thmname{#1}\thmnote{ (#3)}}

\theoremstyle{ntheorem}

\theoremstyle{ndefinition}

\makeatletter
\def\@equationname{equation}

\makeatother

\fancypagestyle{plain}{
\fancyhf{}\fancyfoot[LC,RC]{}
\fancyfoot[LE,RO]{$\thepage$}

}

\UseTips
\SelectTips{eu}{11}

\newdir{ >}{{}*!/-10pt/@{>}}
\newdir{ -}{{}*!/-10pt/@{}}
\newdir{> }{{}*!/+10pt/@{>}}

\makeatletter

\xyletcsnamecsname@{dir4{}}{dir{}}
\xydefcsname@{dir4{-}}{\line@ \quadruple@\xydashh@}
\xydefcsname@{dir4{.}}{\point@ \quadruple@\xydashh@}
\xydefcsname@{dir4{~}}{\squiggle@ \quadruple@\xybsqlh@}
\xydefcsname@{dir4{>}}{\Tttip@}
\xydefcsname@{dir4{<}}{\reverseDirection@\Tttip@}

\xydef@\quadruple@#1{%
	\edef\Drop@@{%
		\dimen@=#1\relax
		\dimen@=.5\dimen@
		\A@=-\sinDirection\dimen@
		\B@=\cosDirection\dimen@
		\setboxz@h{%
			\setbox2=\hbox{\kern3\A@\raise3\B@\copy\z@}%
			\dp2=\z@ \ht2=\z@ \wd2=\z@ \box2
			\setbox2=\hbox{\kern\A@\raise\B@\copy\z@}%
			\dp2=\z@ \ht2=\z@ \wd2=\z@ \box2
			\setbox2=\hbox{\kern-\A@\raise-\B@\copy\z@}%
			\dp2=\z@ \ht2=\z@ \wd2=\z@ \box2
			\setbox2=\hbox{\kern-3\A@\raise-3\B@ \noexpand\boxz@}%
			\dp2=\z@ \ht2=\z@ \wd2=\z@ \box2
		}%
		\ht\z@=\z@ \dp\z@=\z@ \wd\z@=\z@ \noexpand\styledboxz@
	}%
}

\xydef@\Tttip@{\kern2pt \vrule height2pt depth2pt width\z@
	\Tttip@@ \kern2pt \egroup
	\U@c=0pt \D@c=0pt \L@c=0pt \R@c=0pt \Edge@c={\circleEdge}%
	\def\Leftness@{.5}\def\Upness@{.5}%
	\def\Drop@@{\styledboxz@}\def\Connect@@{\straight@{\dottedSpread@\jot}}}
	
\xydef@\Tttip@@{%
	\dimen@=.25\dimen@
 	\B@=\cosDirection\dimen@
	\setboxz@h\bgroup\reverseDirection@\line@ \wdz@=\z@ \ht\z@=\z@ \dp\z@=\z@
	{\vDirection@(1,-1)\xydashl@ \xyatipfont\char\DirectionChar}%
	{\vDirection@(1,+1)\xydashl@ \xybtipfont\char\DirectionChar}%
}

\xydef@\ar@form{
	\ifx \space@\next \expandafter\DN@\space{\xyFN@\ar@form}%
	\else\ifx ^\next \DN@ ^{\xyFN@\ar@style}\edef\arvariant@@{\string^}%
	\else\ifx _\next \DN@ _{\xyFN@\ar@style}\edef\arvariant@@{\string_}%
	\else\ifx 0\next \DN@ 0{\xyFN@\ar@style}\def\arvariant@@{0}%
	\else\ifx 1\next \DN@ 1{\xyFN@\ar@style}\def\arvariant@@{1}%
	\else\ifx 2\next \DN@ 2{\xyFN@\ar@style}\def\arvariant@@{2}%
	\else\ifx 3\next \DN@ 3{\xyFN@\ar@style}\def\arvariant@@{3}%
	\else\ifx 4\next \DN@ 4{\xyFN@\ar@style}\def\arvariant@@{4}%
	\else\ifx \bgroup\next \let\next@=\ar@style
	\else\ifx [\next \DN@[##1]{\ar@modifiers{[##1]}}
	\else\ifx *\next \DN@ *{\ar@modifiers}%
	\else\addLT@\ifx\next \let\next@=\ar@slide
	\else\ifx /\next \let\next@=\ar@curveslash
	\else\ifx (\next \let\next@=\ar@curveinout 
	\else\addRQ@\ifx\next \addRQ@\DN@{\ar@curve@}%
	\else\addLQ@\ifx\next \addLQ@\DN@{\xyFN@\ar@curve}%
	\else\addDASH@\ifx\next \addDASH@\DN@{\defarstem@-\xyFN@\ar@}%
	\else\addEQ@\ifx\next \addEQ@\DN@{\def\arvariant@@{2}\defarstem@-\xyFN@\ar@}%
	\else\addDOT@\ifx\next \addDOT@\DN@{\defarstem@.\xyFN@\ar@}%
	\else\ifx :\next \DN@:{\def\arvariant@@{2}\defarstem@.\xyFN@\ar@}%
	\else\ifx ~\next \DN@~{\defarstem@~\xyFN@\ar@}%
	\else\ifx !\next \DN@!{\dasharstem@\xyFN@\ar@}%
	\else\ifx ?\next \DN@?{\ar@upsidedown\xyFN@\ar@}%
	\else \let\next@=\ar@error
	\fi\fi\fi\fi\fi\fi\fi\fi\fi\fi\fi\fi\fi\fi\fi\fi\fi\fi\fi\fi\fi\fi\fi \next@}

\makeatother

\newcommand{\qfl}{\xymatrix@1@C=10pt{\ar@4 [r] &}}

\renewcommand{\phi}{\varphi}
\renewcommand{\epsilon}{\varepsilon}

\newcommand{\ifthen}[2]{\ifthenelse{#1}{#2}{}}

\definecolor{cyan}{RGB}{175,238,238}

\renewcommand{\leq}{\leqslant}
\renewcommand{\geq}{\geqslant}

\makeatletter
\def\blfootnote{\xdef\@thefnmark{}\@footnotetext}
\makeatother

\newcommand{\sph}{\text{Sph}}

\newcommand{\lbr}{\text{LBr}}
\newcommand{\bra}{\text{Br}}
\newcommand{\seq}{\text{Seq}}
\newcommand{\step}{\text{Step}}

\newcommand{\longsquigglydown}{\xymatrix{{}\ar@{~>}[d]\\ {}}}
\newcommand{\longsquigglyright}{\xymatrix{{}\ar@{~>}[r]&{}}}
\newcommand{\longsquigglyleft}{\xymatrix{{}&{}\ar@{~>}[l]}}
\newcommand{\longsquigglyup}{\xymatrix{{}\\\ar@{~>}[u] {}}}
\newcommand{\longsquiggupleft}{\xymatrix{   {} & {} \\ {} & {} \ar@{~>} [lu]  }}
\newcommand{\longsquiggdownleft}{\xymatrix{   {} & \ar@{~>}[dl]{} \\ {} & {}  }}
\newcommand{\longsquiggdownright}{\xymatrix{   {} \ar@{~>}[dr]{}  &\\ {} & {}  }}
\newcommand{\ov}{\overline}

\newcommand{\col}{\mathsf{Col}}

\newcommand{\di}{$\diamond$}
\newcommand{\mac}{\mathcal}

\newcommand{\iar}{\stackrel{i}{\longrightarrow}}
\newcommand{\ins}{\text{ins}}
\newcommand{\yt}{\text{Tab}(A_}

\newtheorem{prop}[subsubsection]{Proposition}
\newtheorem{thm}[subsubsection]{Theorem}
\newtheorem{defin}[subsubsection]{Definition}
\newtheorem{lemm}[subsubsection]{Lemma}
\newtheorem{cor}[subsubsection]{Corollary}
\newtheorem{exo}[subsubsection]{Example}

\newcommand{\bb}{\mathbf{B}}
\newcommand{\veps}{\varepsilon_i}
\newcommand{\vphi}{\varphi_i}
\makeatletter
\newcommand{\xRrightarrow}[2][]{\ext@arrow 0359\Rrightarrowfill@{#1}{#2}}
\newcommand{\Rrightarrowfill@}{\arrowfill@\equiv\equiv\Rrightarrow}
\newcommand{\xLleftarrow}[2][]{\ext@arrow 3095\Lleftarrowfill@{#1}{#2}}
\newcommand{\Lleftarrowfill@}{\arrowfill@\Lleftarrow\equiv\equiv}
\makeatother
\makeatletter
\newcommand*{\doubleleftarrow}[2]{\mathrel{
		\settowidth{\@tempdima}{$\scriptstyle#1$}
		\settowidth{\@tempdimb}{$\scriptstyle#2$}
		\ifdim\@tempdimb>\@tempdima \@tempdima=\@tempdimb\fi
		\mathop{\vcenter{
				\offinterlineskip\ialign{\hbox to\dimexpr\@tempdima+1em{##}\cr
					\leftarrowfill\cr\noalign{\kern.5ex}
					\leftarrowfill\cr}}}\limits^{\!#1}_{\!#2}}}
\newcommand*{\doublerightarrow}[2]{\mathrel{
		\settowidth{\@tempdima}{$\scriptstyle#1$}
		\settowidth{\@tempdimb}{$\scriptstyle#2$}
		\ifdim\@tempdimb>\@tempdima \@tempdima=\@tempdimb\fi
		\mathop{\vcenter{
				\offinterlineskip\ialign{\hbox to\dimexpr\@tempdima+1em{##}\cr
					\rightarrowfill\cr\noalign{\kern.5ex}
					\rightarrowfill\cr}}}\limits^{\!#1}_{\!#2}}}
\newcommand*{\trTheoremiplerightarrow}[1]{\mathrel{
		\settowidth{\@tempdima}{$\scriptstyle#1$}
		\mathop{\vcenter{
				\offinterlineskip\ialign{\hbox to\dimexpr\@tempdima+1em{##}\cr
					\rightarrowfill\cr\noalign{\kern.5ex}
					\rightarrowfill\cr\noalign{\kern.5ex}
					\rightarrowfill\cr}}}\limits^{\!#1}}}
\makeatother

\newtheorem{manualtheoreminner}{Theorem}
\newenvironment{manualtheorem}[1]{%
	\renewcommand\themanualtheoreminner{#1}%
	\manualtheoreminner
}{\endmanualtheoreminner}

\begin{document}
\thispagestyle{empty}

\begin{center}

\begin{doublespace}
\begin{huge}
{\scshape Coherence for plactic monoids via rewriting theory and crystal structures}
\end{huge}

\bigskip
\hrule height 1.5pt 
\bigskip

\begin{Large}
{\scshape Uran Meha}
\end{Large}
\end{doublespace}

\vfill

\begin{small}\begin{minipage}{14cm}
\noindent\textbf{Abstract -- }
Rewriting methods have been developed for the study of coherence for algebraic objects. This consists in starting with a convergent presentation, and expliciting a family of generating confluences to obtain a coherent presentation -- one with generators, generating relations, and generating relations between relations (syzygies). In this article we develop these ideas for a class of monoids which encode the representation theory of complex symmetrizable Kac-Moody algebras, called plactic monoids. The main tools for this are the crystal realization of plactic monoids due to Kashiwara, and a class of presentations compatible with a crystal structure, called crystal presentations. We show that the compatibility of the crystal structure with the presentation reduces certain aspects of the study of plactic monoids by rewriting theory to components of highest weight in the crystal. We thus obtain reduced versions of Newman's Lemma and Critical Pair Lemma, which are results for verifying convergence of a presentation. Further we show that the family of generating confluences of a convergent crystal presentation is entirely determined by the components of highest weight. Finally we apply these constructions to the finite convergent presentations of plactic monoids of type $A_n$, $B_n$, $C_n$, $D_n$, and $G_2$, due to Cain-Gray-Malheiro.

%
%
%
%
\end{minipage}\end{small}

\vspace{1cm}

\begin{small}\begin{minipage}{12cm}
\renewcommand{\contentsname}{}
\setcounter{tocdepth}{2}
\tableofcontents
\end{minipage}
\end{small}
\end{center}

\clearpage

\tikzset{every tree node/.style={minimum width=1em,draw,circle},
         blank/.style={draw=none},
         edge from parent/.style=
         {draw,edge from parent path={(\tikzparentnode) -- (\tikzchildnode)}},
         level distance=0.8cm}
	
\section{Introduction}

Plactic monoids of classical type are objects that encode the representation theory of finite dimensional complex semisimple Lie algebras. These objects enjoy a rich structure, and may be studied via different approaches: combinatorial, algebraic, geometric, and rewriting theory. The algebraic approach, due to Kashiwara, is centered around a notion of directed labeled graphs called crystals. A key concept in the theory of crystals are the vertices of elements of highest weight. The rewriting approach, due to Cain-Gray-Malheiro, provides finite convergent presentations for the plactic monoids, called \textit{column presentations}. Convergent presentations allow for the study of the monoid via means of rewriting theory, where one can obtain a finite polygraphic resolution of the given monoid, which gives rise to a finitely generated free resolution of the monoid. The first step in this construction consists of identifying the generating relations between relations, also called the generating syzygies, that is a \textit{coherent presentation}. Squier's coherence theorem states how one can obtain a coherent presentation from a given convergent one. In this article, we prove a version of Squier's theorem in the context of plactic monoids, by taking into account the fact that the crystal structure of the monoid interacts nicely with the sets of generators and relations of the column presentation. More precisely we show that Squier's coherent extension is entirely determined by identifying the generating relations between relations at highest weight.

\subsection{Plactic monoids}
Plactic monoids are objects in mathematics that appear in the study of many, seemingly unrelated, problems. Initially Schensted \cite{schensted1961longest} in 1961 uses the notion of \textit{Young tableau} in an algorithmic solution to the problem of finding the longest monotonic subword of a given word $w$ in the alphabet $A_n=\{1,2,\dots,n\}$. In particular he introduces a notion of \textit{insertion} which produces a tableau $P(w)$ from a word $w$, which contains relevant information on subwords of $w$. In 1970, Knuth \cite{knuth1970permutations} identifies the \textit{Knuth relations} as the generating relations for the congruence on $A_n^*$ defined by $P(w_1) = P(w_2)$. Lascoux and Sch\"{u}tzenberger \cite{lascoux1981monoide} define and formally study the free monoid $A_n^*$ quotiented by the Knuth relations, and call it the \textit{plactic monoid of type $A_n$}. This plactic monoid of type $A_n$ finds many applications in algebra and representation theory of type $A_n$ algebraic objects, that is the symmetric group $S_n$, the general lineal group $GL_n(\mathbb{C})$, the special linear algebra $\mathfrak{sl}_n(\mathbb{C})$. For instance Sch\"{u}tzenberger \cite{schutzenberger1977correspondance} uses it to give one of the first proofs of the Littlewood-Richardson, which is a result used to compute products of Schur functions, and tensor products of irreducible representations of $GL_n(\mathbb{C})$ or of the symmetric group $S_n$. A very detailed account of the plactic monoid $Pl(A_n)$ and its combinatorial realization can be found in \cite{lothaire2002algebraic}.
 
With the emergence of the plactic monoid of type $A_n$ as a combinatorial tool which finds application in problems representation theory of type $A_n$, similar objects and constructions were searched for the other types of classical algebraic objects. To this end, Kashiwara \cite{kashiwara1991crystal} introduces a \textit{plactic monoid} associated to any complex symmetrizable Kac-Moody algebra $\mathfrak{g}$. He achieves this in his study of representations of quantum groups Kac-Moody algebras $U_q(\mathfrak{g})$, via a notion of \textit{crystals}. These are certain directed labeled graphs which encode the weight decomposition of integrable representations of $U_q(\mathfrak{g})$ and moreover there in a tensor product of crystals which models the tensor product of representations of $\mathfrak{g}$. Then the the elements of the plactic monoid $Pl(\mathfrak{g})$ may be viewed as the weight spaces of some integrable representation of $\mathfrak{g}$, e.g. an element $p_\mu$ as a $\mu$-weight space of a representation $V_p$. The product of two elements $p_\mu\ast q_\nu$ corresponds to the $(\mu +\nu)$-weight space of $V_p\otimes V_q$. In other words, the plactic monoid $Pl(\mathfrak{g})$ is a model for the representation theory of $\mathfrak{g}$. These objects fulfill the desired goal for the plactic monoids of other classical types: they encode important algebraic information, and moreover lead to proofs of the Littelwood-Richardon rule in larger generality, see the work of Littelmann \cite{littelmann1994littlewood,littelmann1996plactic}.

Of particular interest are the plactic monoids of classical type $Pl(A_n)$, $Pl(B_n)$, $Pl(C_n)$, $Pl(D_n)$, and $Pl(G_2)$, which correspond to the complex Lie algebras $\mathfrak{sl}_n$, $\mathfrak{o}_{2n}$, $\mathfrak{sp}_{2n}$, $o_{2n+1}$ and $G_2$. Drawing from the theory of $Pl(A_n)$, where the elements and product admit suitable combinatorial realizations, similar constructions were searched for the plactic monoids of other types. This enables one carry the study of representation theory of complex Lie algebra to a very rich combinatorial world. This was achieved in the works of Lecouvey \cite{lecouvey2003schensted,lecouvey2002schensted,lecouvey2007part} where he uses adapted notions of tableau and insertion to realize the plactic monoids of classical type. A key notion in phrasing his constructions is that of admissible columns, due to Kashiwara-Nakashima \cite{kashiwara1991crystal}. In particular, Lecouvey gives presentations of the plactic monoids by generators and relations. The work of Cain-Gray-Malheiro \cite{cain2019crystal} extends the work of Lecouvey by producing presentations for $Pl(A_n)$, $Pl(B_n)$, $Pl(C_n)$, $Pl(D_n)$, and $Pl(G_2)$, where the generators are columns, and the generating are oriented and of the form $c_1 c_2 \Longrightarrow (c_1 \leftarrow c_2)$, where $(c_1\leftarrow c_2)$ denotes the Schensted insertion of the column $c_2$ into $c_1$. Their presentations satisfy the properties of \textit{termination} and \textit{confluence}, and this opens up an approach for studying the plactic monoids via rewriting theory.

\subsection{The study of monoids via rewriting theory}

Rewriting theory is a discipline of mathematics that finds applications in studying algebraic objects via certain presentations. In the context of monoids, these presentations are pairs of sets $X = (X_1, X_2)$, with $X_1$ being the generators, and $X_2$ the \textit{oriented} generating relations, i.e. the elements of $X_2$ are of the form $u \Longrightarrow u'$ for $u,u'\in X_1^*$. The monoid presented by $X$ is the quotient $\ov{X}:=X_1^*/\sim_{X_2}$, where $\sim_{X_2}$ is the congruence on $X_1^*$ generated by the relations in $X_2$. To such a presentation we associate a rewriting system
$
\mac{R}(X):= \left( X_1^*\ |\ tuv \Longrightarrow tu'v,\ \text{for } u\Longrightarrow u'\in X_2,\ t,v\in X_1^* \right)
$.
Two non-degeneracy properties of $X$ which facilitate the study of $\ov{X}$ are those of \textit{termination}, which asserts that no infinite rewriting sequences $w_1\Longrightarrow w_2\Longrightarrow w_3\Longrightarrow\cdots$ exist; and \textit{confluence} which states if a word $w\in X_1^*$ can be rewritten to two different words $w_1,w_2$, then $w_1$ and $w_2$ can be rewritten to a common word $w_3$. For a more detailed exposition of rewriting theory, see \cite{klop1990term}.

One advantage of convergent presentations of monoids is that they provide an algorithmic approach for the homotopical and homological study of the given monoid. In the classical work of Eilenberg-Maclane \cite{eilenberg}, for a given monoid $M$ they construct an $\infty$-dimensional CW-complex $Z$ such that its homotopy group $\pi_1(Z)$ is isomorphic to the universal enveloping group of the given monoid $M$. One can then use $Z$ for a the study of homological invariants of $M$. However $Z$ can be very large, though if $M$ comes with a convergent presentation, one can collapse parts of it to obtain a much smaller and computationally-suitable CW-complex $Z'$ which preserves the homotopical structure of $Z$, see Brown \cite{brown}. The work of Guiraud-Malbos in \cite{guiraud2012higher} adopts a categorical language to express these notions. Moreover, they show that a finite convergent presentation $X=(X_1,X_2)$ of a monoid can be extended to an $\infty$-tuple $(X_1,X_2,X_3,\cdots)$ of sets, called a \textit{polygraphic resolution} of $M$. Informally one may view the polygraphic resolution of $M$ as the data of generating $n$-cells via which one can construct any $n$-cell in the categorical version of $Z'$ simply by gluing the generators. As an application of this, one obtains an explicit free resolution of $\mathbb{Z}$ over $\mathbb{Z}M$, and since the construction of the generating $n$-cells in \cite{guiraud2012higher} is very algorithmic, one may use an algorithmic approach for the homological study of $M$. The polygraphic resolution is a cofibrant replacement of the monoid $M$ relative to the folk-model on $(\omega,1)$-categories, see \cite{Lafont}.  We remark here that the approach of studying the homology of an algebraic object via convergent presentations also appears in the work of Anick \cite{anick1986homology} and Hoffbeck-Guiraud-Malbos \cite{hoffbeck} for associative algebras, and the work of Khoroshkin-Dotsenko \cite{DotsenkoKhoroshkin13} for operads.

The first step in constructing a polygraphic resolution from a convergent presentation $X=(X_1,X_2,X_3)$ of a monoid $M$ consists in identifying the set of generating $3$-cells $X_3$, i.e. syzygies, which may be viewed as generating relations between relations. Such a triple $(X_1,X_2,X_3)$ where \textit{coherent presentation} of the $M$. Squier's coherence theorem describes how one can produce a coherent presentation from a convergent presentation $(X_1,X_2)$. The crucial notion in phrasing the theorem is that of the set $\lbr(X)$ of \textit{local branchings} of $X$, consisting of pairs of rewriting steps $(f,g)$ with $f,g\in\mac{R}(X)$. This set is equipped with a partial order $\sqsubset$ generated by the relations $(f,g) \sqsubset (tfv,tgv)$ for $t,v\in X_1^*$. The minimal elements $(f,g)$ of $\lbr(X)$ with respect to $\sqsubset$ are called \textit{critical branchings}. Squier's theorem states that if we take $X_3$ to consist of $3$-cells of the form for each critical branching $(f,g)$
\begin{equation}\label{eq:some 3 cell}
{\xymatrix@!R@R=1em{ & w_1 \ar@2@/^/@{.>}^{f'}[rd] & \\
		w \ar@2@/^/@{->}^f[ur] \ar@2@/_/@{->}_g[rd]& A(f,g) & w_3\\
		& w_2 \ar@2@/_/@{.>}_{g'}[ur] & }}
\end{equation}
then $(X_1,X_2,X_3)$ is a coherent presentation for the monoid $M$. The work of Guiraud-Malbos \cite{guiraud2018polygraphs} provides a deterministic and algorithmic procedure for expliciting these generating 3-cells, via a notion of \textit{normalization strategies}, which produces the pair $(f',g')$ from $(f,g)$. This approach has been employed to describe the coherence relations for different monoids by: Gaussent-Guiraud-Malbos \cite{gaussent2015coherent} for Artin monoids; by Hage-Malbos \cite{hage2017knuth} for the plactic monoid $Pl(A_n)$; by Hage-Malbos \cite{HageMalbos22} for the Chinese monoid.

\subsection{Organization of the article}

In this article we consider introduce a class of presentations that are compatible with a crystal structure, and show that this compatibility leads to a reduced process of obtaining coherence from convergence.

To this end, we begin Section \ref{sec:prelims} by recalling the necessary notions for phrasing the results of this article. The first part of this section, recalls the notion of a convergent presentation as well as two classical results in verifying convergence: Newman's Lemma and the Critical Pair Lemma. We then recall Squier's theorem which asserts how one obtains a coherent presentation from a convergent one, as well as the notion of normalization strategies. We remark that the exposition of these notions is based on \cite{guiraud2012higher} and hence is done in a categorical language. The second part of this section recalls the notions of crystals, and of the plactic monoids of classical type. We give the definition of crystals according to Joseph \cite{joseph}, though we restrict our considerations to a category of crystals denoted $\mathfrak{g}-\textsc{cryst}$ as in the work of Hénriques-Kamnitzer \cite{henriques2006crystals}.  The advantage of this category is that it contains the crystals of classical type, and moreover satisfies a Schur's Lemma. A simplified way to view crystals, due to Stembridge \cite{stembridge2003local} and Cain-Gray-Malheiro \cite{cain2019crystal}, is as directed labeled graphs satisfying certain finiteness conditions. Two key notions for crystals are \textit{connected components}, and vertices of highest weight, which are vertices that no edge enters. We recall how the crystal structure extends to the free monoid generated by the vertices of the crystal $\Gamma$, as well as the \textit{crystal congruence} $\sim_{\Gamma}$ which identifies the isomorphic connected components in $\Gamma^*$. We then recall the notions of plactic monoids as quotients $\Gamma^*/\sim_\Gamma$, and describe the combinatorial realization of $Pl(A_n)$. Finally we recall the finite convergent presentations of $Pl(A_n)$, $Pl(B_n)$, $Pl(C_n)$, $Pl(D_n)$, and $Pl(G_2)$ called \textit{column presentations}, due to Cain-Gray-Malheiro \cite{cain2019crystal}.

In Section \ref{sec:K-rewriting} we consider a more general notion of a crystal congruence $\sim$ on free crystal monoids $\Gamma^*$. We characterize the \textit{crystal monoids} $M$ of the form $\Gamma^*/\sim$ as those monoids equipped with an underlying crystal structure such that the product of $M$ mimicks the tensor product of crystals. We then develop a rewriting theory adapted to studying crystal congruences. This consists in considering \textit{crystal presentations}, that is $X=(X_1,X_2)$ such that $X_1$, $X_2$ are crystals, and the source and target morphisms $s_1,t_1:X_2 \longrightarrow X_1^*$ are crystal morphisms. We show that given such a presentations, the crystal structure of $X$ extends to the sets of rewriting sequences, of local branchings, and of critical branchings of $X$. In particular we prove that termination and confluence of $X$ are entirely determined by the rewriting rules of highest weight, that is $w_1 \Longrightarrow w_2$ such that $w_1$ is a word of highest weight in $X_1^*$. In particular we obtain reduced versions of Newman's Lemma and Critical Pair Lemma via which one can verify convergence of $X$ by considerations at highest weight. Further we show that the column presentations of plactic monoids are crystal presentations.

In Section \ref{sec:crystal structure on squier} we begin by laying the groundwork for proving a reduced Squier's coherence theorem for crystal presentation. We introduce a class of \textit{crystal 2-monoids} $\mac{C}$, whose sets of 1-cells and 2-cells are crystals, and are compatible with the $\star_0$ and $\star_1$ compositions of $\mac{C}$. We then show that the the free 2-monoid $X_2^*$ and the free (2,1)-monoid $X_2^\top$ generated by a crystal presentation $X$ are in fact crystal 2-monoids. Further we consider adapted notions of congruences on 2-monoids and normalization strategies to the context of crystal 2-monoids, and obtain the main result of this article
\begin{quote}
	\begin{manualtheorem}{\ref{thm:squier 2.0}}
		Let $X$ be a convergent crystal 2-polygraph, $r$ a crystal section, and $\sigma$ a crystal normalization strategy for $r$. Let $\Omega$ be the family of generating confluences obtained from $r$ and $\sigma$. Then $\Omega$ is a crystal cellular extension of $X_2^\top$, and the 2-spheres in $\Omega$ are entirely determined by the critical branchings of highest weight of $X$.
	\end{manualtheorem}
\end{quote}
In other words, Theorem \ref{thm:squier 2.0} says that if $(X_1,X_2,X_3)$ is Squier's coherent extension of a convergent crystal presentation $(X_1,X_2)$, then $X_3$ itself is a crystal, and the shape of $A(f,g)$ as in \eqref{eq:some 3 cell} is entirely determined by $A(f^0,g^0)$, where $f^0$ and $g^0$ are the highest weight elements corresponding to $f$ and $g$ respectively. Finally we apply this result to the plactic monoids of classical type $Pl(\Gamma)$, and reduce the identification of their coherent presentations to certain computations with words of highest weights.

\section{Preliminaries}\label{sec:prelims}
In this preliminaries section we recall the notion of coherent presentations of monoids, and the theory of crystals. In the first part of this section, we recall coherent presentations in a categorical language, following the work of Guiraud-Malbos, see \cite{gaussent2015coherent,guiraud2018polygraphs} for a more detailed account. In the second part of this section we recall the notions from crystal theory sufficient for defining \textit{crystal monoids}, see  \cite{hong2002introduction,joseph,kashiwara1990crystalizing,kashiwara1995crystal,kashiwara1991crystal} for a more detailed account.

\subsection{Coherent presentations of monoids}\label{subsec:2 cats and 2 dogs}

In this article we consider presentations of monoids given by generators and certain oriented relations. More formally, a $2$-\textit{polygraph} is a pair of sets $X=(X_1,X_2)$ along with source and target maps
$
X_1^* \doubleleftarrow{s_1}{t_1} X_2.
$ The elements of $X_1$ and $X_2$ are respectively called \textit{generating 1-cells} and \textit{generating 2-cells} of $X$. We often denote the 2-cells in $X_2$ by $s_1(\alpha)\Longrightarrow t_1(\alpha)$. The \textit{monoid presented by} $X$, denoted $\ov{X}$, is the quotient of the free monoid $X_1^*$ by the congruence $\sim_{X_2}$ generated by the relations of the form $s_1(\alpha) \sim t_1(\alpha)$ for $\alpha\in X_2$.

\subsubsection{Termination and normal forms}\label{subsec:termination nad normal forms}
A monoid presented by a 2-polygraph $X$ can be studied via rewriting theory by associating to $X$ the \textit{abstract rewriting system}
$$
\mac{R}(X):=\left(X_1^*\ \ |\ \ us_1(\alpha)v \Longrightarrow ut_1(\alpha)v,\quad \text{for }u,v\in X_1^*,\ \alpha\in X_2\right).
$$
In other words, $\mac{R}(X)$ is simply the set $X_1^*$ along with an oriented relation $\Longrightarrow$ on $X_1^*$ obtained by multiplying the generating 2-cells of $X$ with elements of $X_1^*$. We call a relation $w \Longrightarrow w'$ in $\mac{R}(X)$ a \textit{rewriting step}, and denote the set of rewriting steps by $\step(X)$. A \textit{rewriting sequence} of $X$ is a (possibly infinite) sequence
$$
w_1 \Longrightarrow_1 w_2 \Longrightarrow_2 w_3 \Longrightarrow_3 w_4 \Longrightarrow_4 \cdots
$$
with $w_i\Longrightarrow_i w_{i+1}$ rewriting steps. We denote the set of rewriting sequences by $\text{Seq}(X)$.

A 2-polygraph $X$ is called \textit{terminating} if $\seq(X)$ contains no infinite rewriting sequences. A word $w\in X_1^*$ is called a \textit{normal form} if there exists no rewriting step with the word $w$ as source. We denote the set of normal forms of $X$ by $\text{Nf}(X)$.

\subsubsection{Branchings and confluence}\label{subsec:bra and conf} A \textit{branching} of a 2-polygraph $X=(X_1,X_2)$ is a pair of rewriting sequences $(f,g)$ of $X$ with same source, which we present graphically as
$$
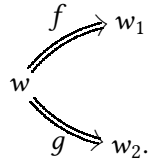
{\xymatrix@!R@R=1em{& w_1 \\ w \ar@2@/^/@{->}^f[ru] \ar@2@/_/@{->}[rd]_g& \\ & w_2.}}
$$
If $f$ and $g$ are rewriting steps, the branching $(f,g)$ is called a \textit{local}. We denote the set of branchings and local branchings of $X$ respectively by $\bra(X)$ and $\lbr(X)$.

A branching $(f,g)\in\bra(X)$ is called \textit{confluent} if there exist rewriting sequences $f',g'\in\seq(X)$ such that
$
s_1(f') = t_1(f),\ s_1(g') = t_1(g),$ and $t_1(f') = t_1(g'),
$
which is graphically presented as 
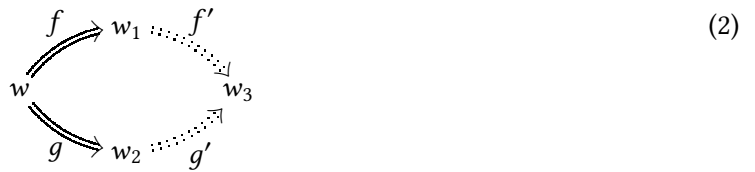
\begin{equation}\label{eq:confdiag}
{\xymatrix@!R@R=1em{ & w_1 \ar@2@/^/@{.>}^{f'}[rd] & \\
		w \ar@2@/^/@{->}^f[ur] \ar@2@/_/@{->}_g[rd]& & w_3\\
		& w_2 \ar@2@/_/@{.>}_{g'}[ur] & }}
\end{equation}
We call the diagram in \eqref{eq:confdiag} a \textit{confluence diagram} of the branching $(f,g)$.
If all (local) branchings of a $2$-polygraph $X$ are confluent, we call $X$ a \textit{(locally) confluent} $2$-polygraph. A 2-polygraph $X$ is called \textit{convergent} if it is terminating and confluent. In that case, any word $w\in X_1^*$ can be rewritten to a unique normal form, denoted $[w]$.

The local branchings of a 2-polygraph $X$ can be grouped into three families:
$$
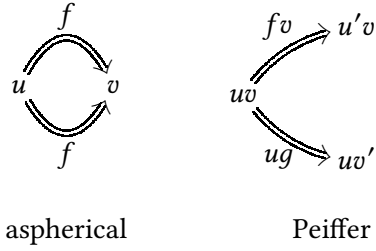
\begin{array}{cc} \xymatrix@!R@R=1em{ & \\ u \ar@2@/^1.5pc/ ^f [r]  \ar@2@/_1.5pc/ _f [r]& v \\ & 
}\quad\qquad & \xymatrix@!R@R=1em{ & u'v\\ uv \ar@2@/^/ [ru] ^{fv} \ar@2@/_/ [rd] _{ug} & \\ & uv'
}\quad\qquad 
\\ & \\
\text{aspherical} \quad\qquad& \text{Peiffer} \quad
\end{array}
$$
with $t,u,v,t',u',v',u''\in X_1^*$ and $f,g \in \text{Step}(X)$, and the remaining local branchings are called \textit{overlapping}. The aspherical branchings are trivial and do not reveal much about the nature of $X$ or the presented monoid $\ov{X}$. Further, the Peiffer branchings trivially admit a confluence diagram via the pair $(u'g,fv')$. Thus the overlapping branchings are the ones that should contain crucial information about $X$ and $\ov{X}$. This will be precised in \ref{subsec:free2monoids} via the notion of free 2-monoids generated by 2-polygraphs.

The set of local branchings $\lbr(X)$ is equipped with an order $\sqsubseteq$ on $\text{LBr}(X)$ generated by the relations
$
(f,g) \sqsubseteq (ufv,ugv)
$
for $(f,g)\in\text{LBr}(X)$ and $u,v\in X_1^*$. An overlapping local branching that is minimal with respect to the order $\sqsubseteq$ is called a \textit{critical branching} or a \textit{critical pair} of $X$. We denote the set of critical branchings of the 2-polygraph $X$ by $\text{Crit}(X)$. A $2$-polygraph may have two types of critical branchings, which we graphically present as:
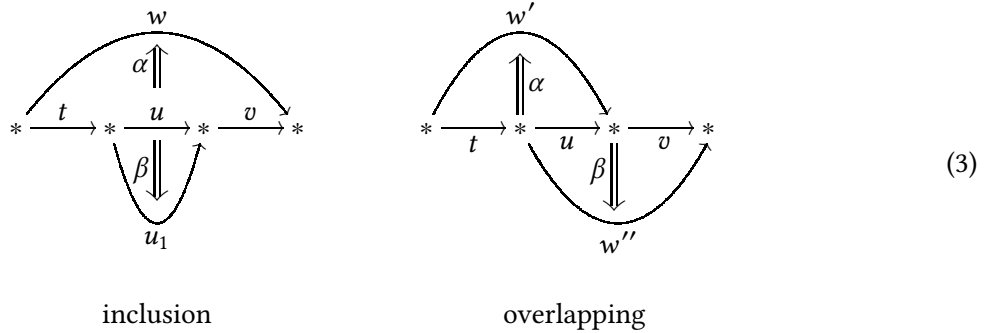
\begin{equation}\label{eq:types of crits}
\begin{array}{ccc}
{\xymatrix{ & & \\ \ast  \ar@/^3.0pc/ [rrr] ^{w}="tgt1" \ar@{->}[r] ^t &\ast \ar@/_3.0pc/[r] _{u_1}="tgt2"\ar@{->}[r]^{u}="src"& \ast \ar@{->}[r] ^v& \ast \\ & &
		\ar@2 "src"!<0pt,10pt>;"tgt1"!<0pt,-10pt> ^\alpha
		\ar@2 "src"!<0pt,-10pt>;"tgt2"!<0pt,15pt> _\beta}} &\qquad\quad& {\xymatrix{ & \ar@{<=} ^\alpha[d] & \\  \ast \ar@/^3.0pc/ [rr] ^{w'}="tgt1" \ar@{->}[r] _t & \ast \ar@/_3.0pc/[rr] _{w''}="tgt2"\ar@{->}[r] _u&  \ast \ar@{->}[r] _v& \ast \ \ \\ & & \ar@{<=}^\beta [u]
}} \\ & \\ \text{inclusion} & \quad\qquad & \text{overlapping}
\end{array}
\end{equation}

We recall here two classical results from rewriting theory on proving the properties of termination and confluence for 2-polygraphs.

\begin{thm} Let $X$ be a 2-polygraph. 
	\begin{itemize}
	 \item[\bf{i)}]{\bf{(\cite{newman1942theories}, Newman's Lemma)}} If $X$ is terminating and locally confluent, then $X$ is confluent.
	 \item[\textbf{ii)}] \textbf{(\cite{baader1999term}, Critical Pair Lemma)} $X$ is locally confluent if and only if its critical branchings are confluent.
	\end{itemize}
\end{thm}

\subsubsection{Free 2-monoids from 2-polygraphs}\label{subsec:free2monoids}
Recall that a 2-monoid is a 2-category $\mac{C}$ with a single 0-cell. If all the 1-cells of $\mac{C}$ are invertible, we call $\mac{C}$ a (2,1)-monoid. We recall here two constructions of free 2-monoids generated by 2-polygraphs.

Let $X=(X_1,X_2)$ be a 2-polygraph. The \textit{free 2-monoid generated by $X$}, denoted $X_2^*$, is defined as follows.
\begin{itemize}
	\item[\textbf{i)}] $X_2^*$ has a single 0-cell denoted by $\ast$;
	\item[\textbf{ii)}] the 1-cells of $X_2^*$ are the words in $X_1^*$, and the $\star_0$ composition is concatenation;
	\item[\textbf{iii)}] the 2-cells are of the form $w\star_0\alpha\star_0 w': wuw' \Longrightarrow wvw'$, for $t,v\in X_1^*$ and $\alpha:u\Longrightarrow u' \in X_2$. We graphically present them as 
		$$
		{\xymatrix@!C@C=4em{
				\ast&
				\ast
				\ar@{<-} [l]_w
				\ar@/^3ex/ [r] ^{u}="src"
				\ar@/_3ex/ [r] _{v}="tgt"
				& \ast 
				\ar@2 "src"!<0pt,-10pt>;"tgt"!<0pt,10pt> ^\alpha & \ar@{<-}[l]_{w'} \ast.
		}}
		$$
		The $\star_1$ composition of 2-cells $w \star_0 \alpha \star_0 w'$ and  $w_1\star_0\beta \star_0 w_1'$, satisfying $w_1t_1(\beta)w_1'=ws_1(\alpha)w'$ is given by
		$$
		(w\star_0 \alpha \star_0 w')\star_1 (w_1\star_0 \beta \star_0 w_1'): wuw' \Longrightarrow wvw' \Longrightarrow w_1zw_1' .
		$$
		Finally the set of 2-cells of is quotiented by the congruence generated by the relations
		\begin{equation}\label{eq:peiffer branchings} \tag{FM1}
		\alpha w v\star_1 u'w \beta \equiv uw\beta \star_1\alpha w v'
		\end{equation}
		for $\alpha:u\Longrightarrow u'$, $\beta:v\Longrightarrow v'$ in $X_2$, and $w\in X_1^*$. 
	
		The $\star_0$ composition of $2$-cells is given by
	\begin{equation}\label{eq:ugly}\tag{FM2}
	\begin{array}{cc}
	&(u_1\alpha_1u_1'\star_1 \cdots \star_1 u_m\alpha_m u'_m)\star_0 (v_1\beta_1 v_1' \star_1\cdots \star_1 v_n\beta_n v_n')\\ &
	= u_1\alpha_1 u_1'v_1s(\beta_1)v_1'\star_1\cdots \star_1 u_m\alpha_m u_m'v_1s(\beta_1)v_1 \star_1\cdots '\\ &
	\star_1 u_m t(\alpha_m)u_m'v_1\beta_1 v_1' \star _1 \cdots \star_1 u_m t(\alpha_m)u_m'v_n\beta_nv_n'.
	\end{array}
	\end{equation}
\end{itemize}

Similarly a 2-polygraph generates a free (2,1)-monoid from a given 2-polygraph, that is one where all the 2-cells are invertible. For this, we consider the polygraph \textit{dual} to $X$, that is the pair $X^-:=(X_1,X_2^-)$ where
$
X_2^-:=\{t_1(\alpha) \Longrightarrow s_1(\alpha) \ |\ \alpha\in X_2\}.
$
Then the \textit{free (2,1)-monoid} generated by $X$, denoted $X_2^\top$, is the free 2-monoid generated by the 2-polygraph $(X_1, X_2 \sqcup X_2^-)$, with set of 2-cells quotiented by the congruence generated by the relations
\begin{equation}\label{eq:equations for x2top} \tag{FM3}
u\alpha v \star_1 u\alpha^-v \equiv 1_{us(\alpha)v}\quad \text{and}\quad u\alpha^-v\star_1u\alpha v \equiv 1_{ut(\alpha)v}.
\end{equation}
We note that by definition, every $2$-cell of $X_2^\top$ is invertible. Moreover two words $u,v\in x_1^*$ represent the same element in the monoid $\ov{X}$ if and only if there exists a $2$-cell $f:u\Longrightarrow v$ in $X_2^\top$. Note that such a $2$-cell is a \textit{zig-zag} rewriting sequence composed of the rewriting steps of $X_2$, i.e. is of the form
$$
u=u_1 \stackrel{\gamma_1}{\Longrightarrow}u_2\stackrel{\gamma_2}{\Longleftarrow} u_3 \stackrel{\gamma_3}{\Longrightarrow}\cdots \stackrel{\gamma_k}{\Longleftarrow} u_{k+1}=v.
$$
with $\gamma_{2i-i}\in \seq(X)$ and $\gamma_{2i}\in\seq(X^-)$ for $1 \leq i \leq k/2$.

\subsubsection{Congruences on 2-monoids}
We recall here the notion of congruences on a 2-monoid.

Let $\mac{C}$ be a $2$-monoid. A \textit{2-sphere} of $\mac{C}$ is a pair $(f,g)$ of $2$-cells of $\mac{C}$ such that
$
s_1(f)= s_1(g)=p$ and $t_1(f)=t_1(g)=q,
$
for some $p,q\in X_1^*$. We also say that $f$ and $g$ are \textit{parallel} $2$-cells. We denote the set of 2-spheres of $\mac{C}$ by $\text{Sph}(\mac{C})$.
A \textit{congruence} on the $2$-monoid $\mac{C}$ is an equivalence relation $\equiv$ on the parallel $2$-cells of $\mac{C}$, i.e. the relations are of the form $f\equiv g$ with $(f,g)\in\text{Sph}(\mac{C})$, which we present graphically as
\begin{equation}\label{eq:2-sphere filled}
\xymatrix@!C@C=6em{
	\ast
	\ar@/^3ex/ [r] ^{p}="src"
	\ar@/_3ex/ [r] _{q}="tgt"
	& \ast,
	\ar@2 "src"!<-15pt,-10pt>;"tgt"!<-15pt,10pt> _f="kiki"
	\ar@2 "src"!<15pt,-10pt>;"tgt"!<15pt,10pt> ^g="titi"
	\ar@3@{-} "kiki"!<12pt,0pt>; "titi"!<-12pt,0pt>
}
\end{equation}
such that for any $2$-cells $h,k$ of $\mac{C}$ such that $h\star_1 f \star_1 k$ is defined, and $1$-cells $w,w'$ of $\mac{C}$ such that $w\star_0 (h\star_1 f\star_1 k) \star_0 w'$ is defined, we have
$$
w \star_0 (h\star_1 f\star_1 k)\star _0 w' \equiv w\star_0 (h\star_1 g\star_1 k)\star_0 w'.
$$
Graphically this means that if $f\equiv g$ as in \eqref{eq:2-sphere filled}, then the following $2$-sphere is also in the relation $\equiv$
$$
\xymatrix@!C@C=6em{
	\ast \ar@{->}[r] ^ w&\ast \ar@/^7ex/[r]_{\ }="manjaro" \ar@/_7ex/[r]^{\ }="kili"
	\ar@/^3ex/ [r] ^{\ }="src"
	\ar@/_3ex/ [r] _{\ }="tgt"
	& \ast\ar@{->}[r] ^{w'} & \ast.
	\ar@2 "src"!<-15pt,-10pt>;"tgt"!<-15pt,10pt> _f="kiki"
	\ar@2 "src"!<15pt,-10pt>;"tgt"!<15pt,10pt> ^g="titi"
	\ar@2 "manjaro"!<0pt,10pt>;"src"!<0pt,-10pt> ^h
	\ar@2 "tgt"!<0pt,10pt>;"kili"!<0pt,-10pt> ^k
	\ar@3@{-} "kiki"!<12pt,0pt>; "titi"!<-12pt,0pt>
}
$$
The \textit{quotient 2-monoid} of $\mac{C}$ by $\equiv$, denoted $\mac{C}/\equiv$, is the 2-monoid defined as follows:
\begin{itemize}
	\item[\textbf{i)}] the $0$-cells and $1$-cells are those of $\mac{C}$;
	\item[\textbf{ii)}] the $2$-cells are the equivalence classes of the $2$-cells of $\mac{C}$ under the congruence $\equiv$.
\end{itemize}

A \textit{cellular extension} of $\mac{C}$ is a collection of 2-spheres of $\mac{C}$. More precisely, it is a set $\Omega$ along with a $2$-source and a $2$-target map
$
\text{Sph}(\mac{C})\doubleleftarrow{t_2}{s_2}\Omega
$
satisfying the globular relations
$
s_1s_2 = s_1 t_2$ and $t_1s_2=t_1t_2.
$
For a cellular extension $\Omega$ of $\mac{C}$, denote by $\equiv_\Omega$ the congruence on $\mac{C}$ generated by the relations
$$
\xymatrix{s_2(\gamma) \ar@3 [r] & t_2(\gamma)}\quad \text{for }\gamma\in\Omega.
$$
We then denote by $\mac{C}/\Omega$ the quotient monoid of $\mac{C}$ by the congruence $\equiv_\Omega$.

\subsubsection{Coherent presentations of monoids}
Next we state the definition of a coherent presentation of a monoid in terms of the categorical language we have used so far. We begin by recalling the following 3-dimensional counterpart to the notion of a 2-polygraph.

A \textit{(3,1)-polygraph} is a triple of sets $X=(X_1,X_2,X_3)$, along with source and target maps
	$$
	X_1^* \doubleleftarrow{t_1}{s_1} X_2^\top \doubleleftarrow{t_2}{s_2} X_3
	$$
	satisfying the \textit{globular relations}
	$s_is_{i+1} = s_it_{i+1}$ and $t_is_{i+1}=t_it_{i+1}$ for $i=0,1$.
Note that for a (3,1)-polygraph $X=(X_1,X_2,X_3)$, the pair $(X_1,X_2)$ along with the source and target maps $s_1,t_1:X_2\longrightarrow X_1^*$ is a 2-polygraph. The notion of coherent presentations is then formulated in terms of (3,1)-polygraphs as follows.
\begin{defin}\ Let $M$ be a monoid, and $\mac{C}$ a $2$-monoid.
	\begin{itemize}
		\item[{\bf i)}]	A cellular extension $\Omega$ of $\mac{C}$ is called \textit{acyclic} if for any $2$-sphere $(f,g)$ of $\mac{C}$ we have $f\sim_\Omega g$.
		\item[{\bf ii)}] A \textit{coherent presentation for}  $M$ is a $(3,1)$-polygraph $X=(X_1,X_2,X_3)$ such that the underlying $2$-polygraph $(X_1,X_2)$ presents $M$, and $X_3$ is an acyclic cellular extension of the free $(2,1)$-category $X_2^\top$.
	\end{itemize} 
\end{defin}

Squier's coherence theorem asserts how one obtains a coherent presentation from a convergent one as follows. Let $X$ be a convergent $2$-polygraph. A \textit{family of generating confluences} of $X$ is a cellular extension $\Omega$ of the free $(2,1)$-category $X_2^\top$ that contains precisely one $3$-cell
\begin{equation}\label{eq:conf diag norm strat}
\xymatrix@!R@R=1em{ & v \ar@2@/^/ [rd] ^{f'} \ar@3 [dd] ^{A_{f,g}}& \\ u \ar@2@/^/ [ru] ^f \ar@2@/_/ [rd] _g & & u' \\ & w \ar@2@/_/ [ru]_{g'} &
}
\end{equation}
for each critical branching $(f,g)$ of $X$.	Such a family always exists, due to $X$ being confluent, though it is not necessarily uniquely determined. We then have the following result due to Squier \cite{squier1994finiteness}.

\begin{thm}[\cite{squier1994finiteness}, Squier's coherence theorem] \label{thm:Squier} Let $X$ be a convergent $2$-polygraph presenting a monoid $M$, and let $\Omega$ be a family of generating confluences for $X$. Then the $(3,1)$-polygraph $(X,\Omega)$ is a coherent presentation for $M$.
\end{thm}

\subsubsection{Normalization strategies}\label{subsec:norm}
The work of Guiraud-Malbos in \cite{guiraud2012higher,guiraud2018polygraphs} extends on the result of Squier by providing a determinstic approach for algorithmically constructing the 3-cells in $\Omega$ via a notion of normalization strategies as follows. 

Let $X=(X_1,X_2)$ be a $2$-polygraph presenting a monoid $M$. Denote by $\pi:X_1^*\longrightarrow M$ the canonical projection, and denote the image $\pi(u)$ simply by $\overline{u}$. Let $s:M\longrightarrow X_1^*$ be a section, that is a map which to each element $u$ of the monoid $M$ associates a representative $\widehat{u}$ of its equivalence class. In other words the map $s$ satisfies $\pi\circ s = \text{id}_M$. We also assume that $\widehat{1}=1$.

A \textit{normalization strategy} for the pair $(M,s)$ is a map
$
\sigma: X_1^*\longrightarrow X_2^\top
$ 
which to each $1$-cell $u\in X_1^*$ associates a 'rewriting path' in $X_2^\top$ from $u$ to the chosen representative $\widehat{u}$. More precisely, for $u\in X_1^*$, $\sigma(u)=\sigma_u$ is a 2-cell in $X_2^\top$
$
\sigma_u: (u\Longrightarrow \widehat{u})
$
such that $\sigma(\widehat{u})=1_{\widehat{u}}$ for all $1$-cells $u\in X_1^*$.	The normalization strategy $\sigma$ is called \textit{left} (resp. \textit{right}) if the following is satisfied for all $u,v\in X_1^*$
$$
\sigma_{uv}=(\sigma_u\star_0 v)\star_1 \sigma_{\widehat{u}v}\quad (\text{resp.}\quad \sigma_{uv}=(u\star_0 \sigma_v)\star_1 \sigma_{u\widehat{v}}).
$$

We recall from \cite{guiraud2018polygraphs} that every $2$-polygraph admits a left (resp. right) normalization strategy. In particular, for a class of 2-polygraphs called reduced, there are two types of practical normalization strategies. We first recall this class of 2-polygraphs.

A $2$-polygraph $X=(X_1,X_2)$ is called \textit{reduced} if it satisfies the following conditions:
	\begin{itemize}
		\item[{\bf i)}] if $u\Longrightarrow v\in X_2$, then $u$ is a normal form for the $2$-polygraph $\big(X_1,X_2\setminus \{u\Longrightarrow v\}\big)$;
		\item[{\bf ii)}] if $u\Longrightarrow v\in X_2$, then $v$ is a normal form for the $2$-polygraph $X$.
	\end{itemize}

Let $X$ be a reduced $2$-polygraph. To a word $u\in X_1^*$ we associate a set
$
\text{Step}(X)_u:=\{f\in\text{Step}(X)\ |\ s(f)=u\}.
$
We define an order on $\step(X)_u$ by setting $f\preceq g$ if $f$ acts more to the left of $u$ than $g$, that is if there exist $t,t',v,v'\in X_1^*$ and $\alpha,\beta\in X_2$ such that
$
f=t\alpha t',\ g=v\beta v'$ and $|t|<|v|.
$
If the polygraph $X$ is finite, then $\text{Step}(X)_u$ is also finite, and since $X$ is reduced, the order $\preceq$ is total, thus there exist a minimal element $\lambda_u$ and a maximal element $\rho_u$ in $\text{Step}(X)_u$. We call these respectively the \textit{leftmost} and \textit{rightmost rewriting steps with source }$u$. Moreover if the polygraph $X$ is terminating, then an iteration of $\lambda$ (resp. $\rho$) gives rise to a normalization strategy called the \textit{leftmost (resp. rightmost) normalization strategy} of $X$ given by
\begin{equation}\label{eq:leftmost}
\sigma_u = \lambda_u\star_1 \sigma_{t(\lambda_u)}\quad \big(\text{resp. }\sigma_u=\rho_u\star_1 \sigma_{t(\rho_u)}\big).
\end{equation}
We can use the leftmost and rightmost normalization strategies to obtain a family of generating confluences for finite convergent $2$-polygraphs $X$ by completing the confluence diagrams of critical branchings of $X$ as in \eqref{eq:conf diag norm strat} with $f' = \sigma_v$ and $g'=\sigma_u$.

\subsection{Crystals and crystal monoids}

The notion of crystals arose in the work of Kashiwara \cite{kashiwara1991crystal} in his study of the representation theory of the quantum groups associated to complex symmetrizable Kac-Moody algebras $\mathfrak{g}$. They are combinatorial objects that encode the data of integrable representations of $\mathfrak{g}$. Here we recall the notion of crystals by following \cite{joseph}, as well as the plactic monoids of classical type $A_n$, $B_n$, $C_n$, $D_n$, and $G_2$ along with their finite convergent presentations due to \cite{cain2019crystal}.
\subsubsection{Algebraic definition of crystals}
Let $\mathfrak{g}$ be a complex symmetrizable Kac-Moody algebra with Cartan datum $(A,\Pi,\Pi^\vee,P,P^\vee)$. Denote its weight lattice by $P$, its set of dominant weights by $P^+$, its set of simple roots by $\{\alpha_i\}_{i\in I}$, and its set of simple coroots by $\{\alpha_i^\vee\}_{i\in I}$. We follow the definition of crystals from Joseph \cite{joseph} by only considering what he calls \textit{normal crystals}.

A \textit{crystal} for $\mathfrak{g}$ is a set $\Gamma$ along with the structure maps
\begin{equation}\label{def:crystal}
\begin{array}{ccc}
\text{wt}:\Gamma \longrightarrow P, \quad & \veps,\vphi:\Gamma \longrightarrow \mathbb{Z}, \quad & e_i,f_i:\Gamma \longrightarrow \Gamma \sqcup\{0\}
\end{array}
\end{equation}
for all $i\in I$, satisfying the following conditions for all $x\in \Gamma$ and $i\in I$:
\begin{itemize}
	\item[\textbf{(C1)}] $\vphi(x)-\veps(x) = \langle \text{wt}(b), \alpha_i^\vee \rangle$,
	\item[\textbf{(C2)}] $\veps(x) = \max\{n\in\mathbb{N}\ |\ e_i^n.x\neq 0\}$ and $\vphi(x) = \max\{n\in\mathbb{N}\ |\ f_i^n.x \neq 0 \}$,
	\item[\textbf{(C3)}] if $e_i.x \neq 0$ then $\text{wt}(e_i.x) = \text{wt}(x) + \alpha_i$, and if $f_i.x\neq 0$ then $\text{wt}(f_i.x) = \text{wt}(x) - \alpha_i$,
	\item[\textbf{(C4)}] $e_i.x = x'$ if an only if $f_i.x'=x$.
\end{itemize}
The map $\text{wt}$ is called the \textit{weight map}, and $e_i,f_i$ are called the \textit{Kashiwara operators}. Condition {\bf ii)} is what makes the crystal \textit{normal}. Note that in this case, the maps $\veps$ and $\vphi$ can be defined in terms of the Kashiwara operators, and in this article when we equip a set with a crystal structure, we shall only specify the definitions of the weight map and the Kashiwara operators. Given an integrable representation $V$ of $\mathfrak{g}$, it admits a crystal $\Gamma$ describing a basis of $V$, where the Kashiwara operators represent the action of the Chevalley generators of $\mathfrak{g}$, see \cite{kashiwara1995crystal}. A crystal $\Gamma$ is called a \textit{highest weight crystal of highest weight} $\lambda\in P^+$ if there exists a unique $x_{\lambda}\in \Gamma$ such that $e_i.x_{\lambda}=0$ for all $i\in I$, and such that any element of $\Gamma$ can be obtained by applying the $f_i$ operators to $x_\lambda$. For given $\lambda\in P^+$, such a crystal is unique and is denoted by $B(\lambda)$.

For two crystals $\Gamma_1, \Gamma_2$, their \textit{tensor product} $\Gamma_1\otimes \Gamma_2$ is equal to $\Gamma_1\times \Gamma_2$ as a set, has weight map given by $\text{wt}(xy)=\text{wt}(x) + \text{wt}(y)$, and the Kashiwara operators are on $u\otimes v \in \Gamma_1\otimes \Gamma_2$ via
\begin{equation}\label{eq:def of e}
e_i.(u\otimes v)=\left\{\begin{array}{cc}
(e_i.u)\otimes v & \text{if }\varphi_i(u)\geq \varepsilon_i(v),\\
u\otimes (e_i.v) & \text{if } \varphi_i(u) < \varepsilon_i(v),
\end{array}\right. ,\qquad f_i.(u\otimes v)=\left\{\begin{array}{cc}
(f_i.u)\otimes v & \text{if }\varphi_i(u)> \varepsilon_i(v),\\
u\otimes (f_i.v) & \text{if } \varphi_i(u) \leq \varepsilon_i(v).
\end{array}\right.
\end{equation}
Given a crystal $\Gamma$, we note that the free monoid $\Gamma^*$ inherits a crystal structure by viewing it as $\Gamma^*=\bigsqcup_{n=0}^\infty \Gamma^{\otimes n}$. We call $\Gamma^*$ the \textit{free crystal monoid generated by $\Gamma$}. For a word $w\in \Gamma^*$, we denote by $\bb_{\Gamma^*}(w)$ the \textit{connected component of $w$ in $\Gamma^*$}, which consists of all the vertices of $\Gamma^*$ obtained by successively applying the Kashiwara operators $f_i$ and $e_i$ to $w$.  Evidently $\bb_{\Gamma^*}(w)$ is also a crystal with the restricted structure maps.

A \textit{morphism} of crystals, called a \textit{strict morphism} in literature \cite{hong2002introduction,joseph,kashiwara1995crystal}, is a map $F : \Gamma_1\longrightarrow \Gamma_2$ that commutes with all the structure maps $\text{wt}, \veps,\vphi, e_i, f_i$, including the $0$ element. For $\lambda,\mu\in P^+$ we have the following analogue of Schur's Lemma, see \cite{hong2002introduction}
$$
\hom(B(\lambda),B(\mu)) = \left\{\begin{array}{ll}
\text{id}_{B(\lambda)} & \quad\text{if } \lambda = \mu,\\
\varnothing &\quad \text{otherwise}.
\end{array}\right.
$$
In \cite{henriques2006crystals} Hénriques-Kamnitzer consider the \textit{category of crystals} $\mathfrak{g}\textsc{-cryst}$ whose objects are normal crystals $\Gamma$ such that for any $x\in\Gamma$ we have $\bb_{\Gamma}(x)\cong B(\lambda)$ for some $\lambda\in P^+$, and the morphisms in $\mathfrak{g}\textsc{-cryst}$ are strict morphisms of crystals. This category is monoidal, that is closed under taking tensor products, hence in particular $\Gamma \in \mathfrak{g}\textsc{-cryst}$ implies $\Gamma^*\in\mathfrak{g}\textsc{-cryst}$. In this article the word crystal will refer to an object of $\mathfrak{g}\textsc{-cryst}$.

While the definition of crystals is heavily linked to its origins in the study of Lie algebras, one might speak of crystals in a more graph theoretic way. This approach is developed by Stembridge in \cite{stembridge2003local}, in his work of axiomatizing the crystals which arise from integrable representations. Further this approach was refined by Cain-Gray-Malheiro \cite{cain2019crystal} to give a graph-theoretic notion of crystals, sufficient for defining a notion of crystal monoid. In this approach, one speaks of crystals as directed $I$-labeled graphs $\Gamma$ satisfying the following two conditions:
\begin{itemize}
	\item[{\bf (P1)}] for any $x\in V(\Gamma)$ and $i\in I$, there exists at most one edge $e$ with source (resp. target) $x$ and label $i$;
	\item[{\bf (P2)}] for any $i\in I$, there exists no infinite directed path in $\Gamma$ with edges labeled by $i$.
\end{itemize}
Crystals fit into this scheme via the identification
$$
x \iar y \quad \text{if } e_i.y=x\quad (\text{equivalently if } f_i.x=y).
$$
The condition \textbf{(P1)} corresponds to the Kashiwara operators, and condition \textbf{(P2)} corresponds to the \textit{normality} condition for crystals. In particular the quantity $\veps(x)+\vphi$ gives the length of the longest $i$-labeled path in $\Gamma$ passing through the vertex $x$. Further, Cain-Gray-Malheiro consider an abstracted weight map $\text{wt}:\Gamma\longrightarrow P$, where $P$ is an ordered poset such that if $e_i.x \neq 0$ (resp. $f_i.x\neq 0$), then $\text{wt}(e_i.x) > \text{wt}(x)$ (resp. $\text{wt}(f_i.x)<\text{wt}(x)$.

In this article we stick to the formality of crystals in terms of the axioms \textbf{(C1)}-\textbf{(C4)}, though we illustrate certain notions and results via the graph theoretic approach due to Cain-Gary-Malheiro \cite{cain2019crystal}.
\subsubsection{Crystal monoids}\label{subsub:crystal monoids} Let $\mathfrak{g}$ be a complex symmetrizable Kac-Moody algebra, $\Gamma$ a crystal for $\mathfrak{g}$, and consider the free monoid $\Gamma^*$ along with the crystal structure on it. A \textit{crystal congruence} $\sim$ on $\Gamma^*$ is such that if $w_1\sim w_2$, then $\bb(w_1)\cong \bb(w_2)$. In \cite{cain2019crystal}, Cain-Gray-Malheiro consider the equivalence relation $\sim_\Gamma$ defined by setting $w_1\sim_\Gamma w_2$ if $\bb(w_1)\cong\bb(w_2)$. They show that this is a congruence on $\Gamma^*$, and they define the \textit{crystal monoid associated to $\Gamma$} as the quotient
$
C(\Gamma)=\Gamma^* / \sim_\Gamma.
$
Clearly $\sim_\Gamma$ is the largest crystal congruence on $\Gamma^*$.

\subsubsection{Crystal bases of classical type}\label{subsec:classical crystals}

For $\mathfrak{g}$ a complex semisimple Lie algebra of classical type $A_n$, $B_n$, $C_n$, $D_n$, and $G_2$, the following directed labeled graphs are the crystals corresponding to the fundamental representations $\Lambda_1(\mathfrak{g})$
\begin{equation}\label{eq:crystal A}
\begin{array}{lc}
A_n: & 1\stackrel{1}{\longrightarrow}2\stackrel{2}{\longrightarrow}\cdots \stackrel{n-2}{\longrightarrow}n-1 \stackrel{n-1}{\longrightarrow} n
\end{array}
\end{equation}

\begin{equation} \label{eq:crystal B}
\begin{array}{cc}
B_n: & 1\stackrel{1}{\longrightarrow}2\stackrel{2}{\longrightarrow}\cdots \stackrel{n-1}{\longrightarrow} n\stackrel{n}{\longrightarrow} 0 \stackrel{n}{\longrightarrow} \bar{n}\stackrel{n-1}{\longrightarrow}\cdots \stackrel{2}{\longrightarrow}\bar{2}\stackrel{1}{\longrightarrow} \bar{1}
\end{array}
\end{equation}

\begin{equation}\label{eq:crystal C}
\begin{array}{cc}
C_n: & 1\stackrel{1}{\longrightarrow}2\stackrel{2}{\longrightarrow}\cdots \stackrel{n-1}{\longrightarrow} n \stackrel{n}{\longrightarrow} \bar{n}\stackrel{n-1}{\longrightarrow}\cdots \stackrel{2}{\longrightarrow}\bar{2}\stackrel{1}{\longrightarrow} \bar{1}
\end{array}
\end{equation}

\begin{equation}
\begin{array}{cc} \xymatrix{
	& & & & & n \ar@1 [rd] ^n & & &\\ D_n: &
	1\ar@1 [r] ^1 & 2\ar@1 [r] ^2 & \cdots\cdots \ar@1 [r] ^{n-2} &n-1 \ar@1 [ru] ^{n-1} \ar@1 [rd]_n & & \overline{n-1}\ar@1 [r] ^{n-2}& \cdots\cdots \ar@1 [r] ^2 & \overline{2}\ar@1 [r] ^1& \overline{1}\\
	& & & & & \overline{n} \ar@1 [ru]_{n-1}& & & }
\end{array}
\end{equation}

\begin{equation}\label{eq:crystal G}
\begin{array}{cc}
G_2: 1 \stackrel{1}{\longrightarrow}2 \stackrel{2}{\longrightarrow} 3 \stackrel{1}{\longrightarrow} 0 \stackrel{1}{\longrightarrow} \bar{3} \stackrel{2}{\longrightarrow} \bar{2} \stackrel{1}{\longrightarrow} \bar{1}
\end{array}
\end{equation}
and are respectively called the \textit{crystal bases of classical type}. They are equipped with a weight map to the weight lattice $P$ of the corresponding Lie algebra. The exact definitions of the weight maps are not relevant for the contents of this article, though here we recall the weight map for $A_n$. The weight lattice $P$ is $n$-dimensional, and the weight map is defined as
$
\text{wt}(i) = (0,\dots, i,\dots, 0),
$
that is the $j$-th coordinate is given by the Kronecker delta $\delta_{ij}$.
The crystal monoids associated to the crystal bases $\Gamma=A_n, B_n, C_n, D_n, G_2$ are called the \textit{plactic monoids of type $\Gamma$}, and are denoted by $Pl(\Gamma)$. Clearly the crystal structure of $\Gamma^*$ descends to the plactic monoid $Pl(\Gamma)$.

\subsubsection{Combinatorial realization of plactic monoids} The plactic monoids of classical type are realized combinatorially, with their elements parameterized by notions of tableaux, and the product via certain insertion algorithms. Since the intricacies of these realizations are not crucial for the results of this paper, we shall recall the constructions in detail for type $A_n$, and in more general terms for the other types $B_n, C_n, D_n, G_2$. For a more detailed exposition see \cite{lothaire2002algebraic}.

The elements of $Pl(A_n)$ are parameterized by \textit{semi-standard Young tableaux}, which are a collection of left-adjusted rows of boxes, filled with entries from $A_n$ so that the entries are non-decreasing along rows, and increasing along columns. For instance the following is a semi-simple Young tableau for $A_5$
\begin{equation}\label{eq:ytab}
\begin{ytableau}
1 & 1 & 2 & 3\\
2 & 3 & 3\\
4 & 5
\end{ytableau}.
\end{equation}
The set of semistandard Young tableaux is denoted $\text{Tab}(A_n)$, and we have a \textit{reading map} $l_c:\text{Tab}(A_n) \longrightarrow A_n^*$, which reads the columns of a tableau $T$ from right to left, and top to bottom. For the tableau in \eqref{eq:ytab} we have $l_c(T)=3\ 23\ 135\ 124$. As sets we have $Pl(A_n) = \text{Tab}(A_n)$.

The product of $Pl(A_n)$ is parameterized by \textit{Schensted's insertion algorithm}. It is an algorithm that takes as input a word $w\in A_n^*$ and produces a tableau $P_c(w)\in\text{Tab}(A_n)$, which is a representative of the equivalence class of $w$ in $Pl(A_n)$. It is defined as a map $
\ins_c: A_n \times \yt n) \longrightarrow \yt n)
$
with $\ins_c(x,T)$ defined as follows

\begin{itemize}
	\item[1.] for $x\in A_n$, set $\ins_c\left(x,\boxed{\textcolor{white}{0}}\right)=\boxed{x}$;
	\item[2.] for $x\in A_n$, $T\in \yt n)$, denote by $x_1,x_2,\dots, x_k$ the entries of the first column of $T$, and by $T_1$ the tableau obtained by deleting the first column from $T$. Then $\ins_c(x,T)=T'$ where $T'$ is determined as follows:
	\begin{itemize}
		\item[2a.] if $x > x_k$, $T'$ is obtained by adding a box $\boxed{x}$ at the end of the first column of $T$;
		\item[2b.] if $x\leq x_k$, let $j$ be minimal with the property $x\leq x_{j}$. Then the first column of $T'$ is obtained by replacing $x_{j}$ with $x$, and its remaining columns are given by the tableau $\ins_c(x_{j},T_1)$.
	\end{itemize}
\end{itemize}
Via $\ins_c$ we have a map $P_c:A_n^*\longrightarrow \yt n)$ defined
inductively on the length of words of $A_n^*$ by setting $P_c(x)=\boxed{x}$ for $x\in A_n$, and $P_c(wx)=\ins_c(x,P_c(w))$ for $x\in A_n$ and $w\in A_n^*$. Then the product $\ast$ of the monoid $Pl(A_n)$ is encoded via the map $P_c$ as
\begin{equation}\label{eq:prod of tab}
T_1\ast T_2 = P_c(l_c(T_2)l_c(T_1)).
\end{equation}

Similar constructions exits for the other classical types as well due to the work of Lecouvey \cite{lecouvey2002schensted,lecouvey2003schensted,lecouvey2007part}. We summarize this as follows. Let $\Gamma=A_n,B_n,C_n,D_n,G_2$. There exists a notion of \textit{tableaux} in type $\Gamma$, with their set denoted by $\text{Tab}(\Gamma)$. As sets we have $Pl(\Gamma)=\text{Tab}(\Gamma)$. Moreover, there exists a reading map $l_c:\text{Tab}(\Gamma)\longrightarrow \Gamma^*$, and an \textit{insertion} map $P_c:\Gamma^* \longrightarrow \text{Tab}(\Gamma)$ such that the product of $T_1,T_2\in\text{Tab}(\Gamma)$ is computed via \eqref{eq:prod of tab}.

Since $Pl(\Gamma)=\text{Tab}(\Gamma)$, the set of tableau admits a crystal structure. We summarize here the interaction between the combinatorial and crystal realizations of plactic monoids.

\begin{prop}[\cite{lecouvey2007part}]\label{prop:P_c crystal morphism}
	The reading map $l_c: \text{Tab}(\Gamma)\longrightarrow \Gamma^* $ and the insertion map $P_c:\Gamma^* \longrightarrow \text{Tab}(\Gamma)$ are morphisms of crystals.
\end{prop}

\subsubsection{Column presentation of plactic monoids}\label{subsec:col pres} Let $\Gamma=A_n,B_n,C_n,D_n,G_2$. The work of Cain-Gray-Malheiro \cite{cain2019crystal} makes use of the insertion map $P_c$ on a set of words in $\Gamma^*$, called \textit{admissible columns}, to give finite convergent presentations for $Pl(\Gamma)$. In type $A_n$, admissible columns coincide with the notion of \textit{columns}, which are simply words $w=x_1x_2\cdots x_k$ with $x_1<x_2<\dots<x_k$, while in other types the notion is more involved. The relevant information is that the set $\col(\Gamma)_1$ of admissible columns in type $\Gamma$ satisfies the two following conditions
\begin{itemize}
	\item[\textbf{i)}] $\col(\Gamma)_1$ is a crystal with the restricted structure maps of $\Gamma^*$,
	\item[\textbf{ii)}] for $c\in\col(\Gamma)_1$ we have $P_c(c)=c$.
\end{itemize}

In order to construct finite convergent presentations of $Pl(\Gamma)$ in \cite{cain2019crystal}, the authors first prove the following result.
\begin{prop}[\cite{cain2019crystal}, 2-column lemmata] \label{prop:2 col lemmata}Let $\Gamma=A_n,B_n,C_n,D_n$ and $c_1,c_2$ be two admissible columns in type $\Gamma$. Then the tableau $P_c(c_1c_2)$ consists of at most two columns. If $\Gamma=G_2$, then $P_c(c_1c_2)$ consists of at most three columns.
\end{prop} 
They use this result to define the \textit{column presentation} of type $\Gamma$ as the 2-polygraph
$
\mathsf{Col}(\Gamma)=(\mathsf{Col}(\Gamma)_1,\mathsf{Col}(\Gamma)_2),
$
where
$$
\mathsf{Col}(\Gamma)_2 = \big\{c_1c_2\Longrightarrow P_c(c_1c_2)\ |\ \text{if }c_1c_2 \neq P_c(c_1c_2)\big\},
$$
and prove the following.
\begin{thm}[\cite{cain2019crystal}]\label{thm:col pres is convergent}
	Let $\Gamma=A_n,B_n,C_n,D_n, G_2$. The $2$-polygraph $\col(\Gamma)$ is a finite convergent reduced presentation of $Pl(\Gamma)$.
\end{thm}
This opens up a direction of applying the techniques of rewriting theory to these presentations, as a way to obtain information on the plactic monoids of classical type.

\section{Rewriting theory for crystal congruences}\label{sec:K-rewriting}
In this section we consider a notion of crystal congruence, which generalizes the congruence $\sim_\Gamma$ in \ref{subsub:crystal monoids}. We thus obtain a class of crystal monoids, and introduce a notion of a 2-polygraph compatible with a crystal structure. We show that this compatibility extends the crystal structure to the sets rewriting sequences and confluence diagrams of the 2-polygraph. This interaction reduces the verification of the rewriting properties of termination and confluence to rewriting rules of highest weight. We thus obtain reduced variants of Newman's Lemma and the Critical Pair Lemma.

\subsection{Crystal congruences and crystal monoids}
%

We begin by introducing a notion of a congruence which is compatible with the crystal structure.
\begin{defin}\label{def:crystal congruences}
	Let $\Gamma$ be a crystal. A \textit{crystal congruence} $\sim$ on $\Gamma^*$ is such that $w_1 \sim w_2$ implies $\bb_{\Gamma^*}(w_1)\cong \bb_{\Gamma^*}(w_2)$ for any $w_1,w_2\in \Gamma^*$. The monoid $M=(\Gamma^*/\sim)$ presented by the congruence is called a \textit{crystal monoid}.
\end{defin}

Let $\Gamma$ be a crystal, $\sim$ a crystal congruence on $\Gamma^*$, and $M$ the monoid presented by $\sim$. The crystal structure of $\Gamma^*$ descends to $M$ as follows. Let $m\in M$, and $w\in\Gamma^*$ be a representative of $m$, that is $[w]=m$. We define the structure maps on $M$ by setting:
$$
\text{wt}(m)=\text{wt}(w), \quad e_i.m=[e_i.w]\quad (\text{resp, }f_i.m=[f_i.w])
$$ 
if $e_i.w\neq 0$ (resp. $f_i.w\neq 0$), and otherwise $e_i.m=0$ (resp. $f_i.m=0$). Note that the definitions of the structure maps are independent of the choice of representative. Indeed, if $m=[w_1]=[w_2]$ we have that $w_1\sim w_2$ which implies $\bb(w_1)\cong \bb(w_2)$, hence the values of the structure maps of $\Gamma^*$ are identical on $w_1$ and $w_2$. This justifies calling $M$ a crystal monoid, and moreover this argument shows that the canonical projection $\pi:\Gamma^* \longrightarrow M$ is a crystal morphism. Recall the full crystal congruence $\sim_\Gamma$ from \ref{subsub:crystal monoids}. Then since any crystal congruence $\sim$ is a subcongruence of $\sim_\Gamma$, we obtain the following

\begin{prop}\label{prop:crystal monoids}
	Let $\Gamma$ be a crystal, and $\sim$ a crystal congruence on $\Gamma^*$, presenting a monoid $M$. We then have surjective crystal morphisms of monoids
	$
	\Gamma^* \stackrel{\pi}{\longrightarrow} M \stackrel{q}{\longrightarrow} \mac{C}(\Gamma).
	$
\end{prop}
Before we state the next result, we recall a direct consequence of the definition of the tensor products of crystals.

\begin{lemm}\label{lem:aux}
	Let $\Gamma_1, \Gamma_2, \Gamma_1',\Gamma_2'$ be crystals, and $x_1 \in \Gamma_1$, $x_2\in\Gamma_2$, $x_1'\in \Gamma_1'$, and $x_2'\in \Gamma_2'$. Suppose that $\bb_{\Gamma_1}(x_1)\cong \bb_{\Gamma_1'}(x_1')$ and $\bb_{\Gamma_2}(x_2)\cong \bb_{\Gamma_2'}(x_2')$. Then
	$$
	\bb_{\Gamma_1\otimes \Gamma_2}(x_1\otimes x_2) \cong \bb_{\Gamma_1'\otimes \Gamma_2'}(x_1'\otimes x_2').
	$$
\end{lemm}
In other words, the connected components in the tensor product of two crystals are entirely determined by the connected components in the two crystals. The following result characterizes crystal monoids as those monoids with an underlying crystal structure whose product models the tensoring of crystals.

\begin{thm}\label{thm:char of crystal monoids}
	Let $M$ be a monoid with an underlying crystal structure. Then $M$ is a crystal monoid if and only if for all $m_1,m_2\in M$ we have
	$
	\bb_M(m_1m_2)\cong \bb_{M\otimes M}(m_1\otimes m_2).
	$
\end{thm}
\begin{proof}
	Let $M$ be a crystal monoid, say $M= (\Gamma^*/\sim)$ for some crystal $\Gamma$ and a crystal congruence $\sim$ on $\Gamma^*$. Let $w_1,w_2\in\Gamma^*$ be representatives of $m_1$ and $m_2$ respectively. Since $w_1w_2$ is a representative of $m_1m_2$, and $\Gamma^* = \bigsqcup_{n=0}^\infty \Gamma^{\otimes n}$, we obtain
	\begin{equation}\label{eq:11}
	\bb_M(m_1m_2)\cong \bb_{\Gamma^*}(w_1w_2) \cong \bb_{\Gamma^*}(w_1\otimes w_2).
	\end{equation}
	Let $A_1=\bb_{\Gamma^*}(w_1)$ and $A_2=\bb_{\Gamma^*}(w_2)$ be the connected components of $w_1$ and $w_2$ in $\Gamma^*$. By Lemma \ref{lem:aux} we obtain
	\begin{equation}\label{eq:12}
	\bb_{\Gamma^*}(w_1\otimes w_2) \cong \bb_{A_1\otimes A_2}(w_1\otimes w_2).
	\end{equation}
	Further, by definition of crystal monoids we have $A_1'=\bb_{M}(m_1)\cong A_1$ and $A_2'=\bb_M(m_2)$, which gives us
	\begin{equation}\label{eq:13}
	\bb_{A_1\otimes A_2} (w_1\otimes w_2) \cong \bb_{A_1'\otimes A_2'}(m_1\otimes m_2).
	\end{equation}
	Finally, since the connected component $\bb_{M\otimes M}(m_1\otimes m_2)$ is entirely determined by $A_1'$ and $A_2'$, by combining \eqref{eq:11},\eqref{eq:12}, and \eqref{eq:13}, we obtain
	\begin{equation}\label{eq:ccc}
	\bb_M(m_1m_2)\cong \bb_{M\otimes M}(m_1\otimes m_2)
	\end{equation}
	which is what we wanted to show.
	
	Conversely, let $M$ be a monoid with an underlying crystal structure such that \eqref{eq:ccc} is satisfied for all $m_1,m_2\in M$. We claim that for any $m_1,\dots,m_k\in M$ we have
	\begin{equation}\label{eq:14}
	\bb_M(m_1\cdots m_k) \cong \bb_{M^{\otimes k}}(m_1\otimes m_2\otimes \cdots \otimes m_k).
	\end{equation}
	For $k\leq 2$ this is the hypothesis of the theorem. We prove it for $k>2$ by induction. Indeed, suppose that \eqref{eq:14} holds for $k$. Consider $m_1,\dots, m_k, m_{k+1}\in M$, and set $m'=m_2\cdots m_{k+1}$, $A_1=\bb_M(m_1)$, and $A'=\bb_{M}(m')$. Then by \eqref{eq:ccc} and Lemma \ref{lem:aux} we have
	$$
	\bb_M(m_1\cdots m_{k+1}) \cong \bb_M(m_1 m') \cong \bb_{M\otimes M}(m_1\otimes m') \cong \bb_{A_1\otimes A'}(m_1\otimes m').
	$$
	By inductive hypothesis, we have $A'\cong \bb_{M^{\otimes k}}(m_2\otimes \cdots m_{k+1})$, and again by Lemma \ref{lem:aux} we obtain
	$$
	\bb_{A_1\otimes A'}(m_1\otimes m') \cong \bb_{A_1 \otimes M^{\otimes k}}(m_1\otimes(m_2\otimes\cdots m_{k+1})) \cong \bb_{M^{\otimes k}}(m_1\otimes \cdots \otimes m_{k+1})
	$$
	which proves \eqref{eq:14}. Next we prove that $M$ is indeed a crystal monoid. Since $M$ is a crystal, we can construct the free crystal monoid $M^*$, along with the canonical projection $r:M^*\longrightarrow M$ given by $r(m_1\otimes \cdots m_k) = m_1\cdots m_k\in M$. From \eqref{eq:14} we have that $r$ is a crystal morphism. The trivial presentation of $M$ is given by $M^*/\sim$ where $\sim$ is given by
	$
	w_1\sim w_2 \quad \text{if}\quad r(w_1)=r(w_2).
	$
	From \eqref{eq:14} we see that $\sim$ is a crystal congruence, hence $M$ is a crystal monoid, which is what we wanted to show.
\end{proof}

Next we introduce a notion of 2-polygraph with all of it defining data coming from the category of crystals.

\begin{defin}\label{def:crystal 2 poly}
	A crystal 2-polygraph is a $2$-polygraph $X=(X_1,X_2)$ such that
	\begin{itemize}
		\item[{\bf i)}]$X_1$ and $X_2$ are crystals,
		\item[{\bf ii)}] the source and target maps $s,t:X_2\longrightarrow X_1^*$ are morphisms of crystals.
	\end{itemize}
\end{defin}
This definition is the adaptation of the notion of 2-polygraphs to the category of crystals. We interpret the second condition graphically as follows:
\begin{equation}\label{eq:katror1}
{\xymatrix{ w \ar@{=>}^{\alpha}[r] \ar@{->}_i[d] & w_1 \ar@{.>}^i[d] \\ f_i.w \ar@{:>}_{f_i.\alpha}[r] & f_i.w_1}}
\end{equation}
That is for a rewriting rule $w \stackrel{\alpha}{\Longrightarrow} w_1$, and $i\in I$ such that $f_i.w$ is defined, then so is $f_i.w_1$ and we have a rewriting rule $f_i.\alpha$.

Next we show that crystal 2-polygraphs generate crystal congruences.

\begin{prop}
	Let $X$ be a crystal 2-polygraph. Then the congruence $\sim_{X_2}$ generated by $X_2$ is a crystal congruence, and hence $X$ presents a crystal monoid.
\end{prop}
\begin{proof}
	Recall the abstract rewriting system $\mac{R}(X)$ associated to $X$, as \ref{subsec:termination nad normal forms}. In subsection \ref{subsec:free2monoids} we remark that two words $u,v\in X_1^*$ are $\sim_{X_2}$-congruent if there exists a zig-zag rewriting sequence of the form
	$$
	u=u_1 \stackrel{\gamma_1}{\Longrightarrow}u_2\stackrel{\gamma_2}{\Longleftarrow} u_3 \stackrel{\gamma_3}{\Longrightarrow}\cdots \stackrel{\gamma_k}{\Longleftarrow} u_{k+1}=v,
	$$
	where $\gamma_i$ are rewriting sequences. Thus to prove that $\sim_{X_2}$ is a crystal congruence, it suffices to show that for a rewriting step $tuv \stackrel{t\alpha v}{\Longrightarrow} tu'v$ in $\mac{R}(X)$ we have $\bb_{X_1^*}(tuv) \cong \bb_{X_1^*}(tu'v)$. Since $u \Longrightarrow u'$ implies $\bb_{X_1^*}(u)\cong \bb_{X_1^*}(u')$ by Definition \ref{def:crystal 2 poly} \textbf{ii)}, for the full crystal congruence $\sim_{X_1}$ on $X_1^*$ we have that $u\sim_{X_1} u'$. Since $\sim_{X_1}$ is a crystal congruence, we have that $\bb_{X_1^*}(tuv) \cong \bb_{X_1^*}(tu'v)$. This shows that $\sim_{X_2}$ is indeed a crystal congruence, and the monoid $\ov{X}$ presented by $X$ is a crystal monoid.
\end{proof}
Next we show that the column presentations of the plactic monoids of classical types are crystal 2-polygraphs.

\begin{thm}\label{thm:col pres are good}
	Let $\Gamma=A_n,B_n,C_n,D_n,G_2$. Then the column presentation $\col(\Gamma)$ is a finite convergent reduced crystal 2-polygraph.
\end{thm}
\begin{proof}
	In \ref{subsec:col pres} we have remarked that the set of generators $\col(\Gamma)_1$ form a crystal. The set of generating relations $\col(\Gamma)_2$ consists of the rewriting rules $c_1c_2 \Longrightarrow P_c(c_1c_2)$ for $c_1,c_2\in\col(\Gamma)_1$ such that $c_1c_2\neq P_c(c_1c_2)$. Note that if $e_i.(c_1c_2)\neq 0$, since $P_c$ is a crystal morphism by Proposition \ref{prop:P_c crystal morphism}, we have that
	$$
	c_1c_2 = P_c(c_1c_2) \quad \text{if and only if} \quad e_i.(c_1c_2) = P_c(e_i.(c_1c_2)).
	$$
	This and the analogous statement for $f_i$ allows shows that $\col(\Gamma)_2$ naturally inherits the crystal structure from $\Gamma^*$ via
	$$
	\col(\Gamma)_2 = \bigcup_{c_1c_2\neq P_c(c_1c_2)} \bb_{\Gamma^*}(c_1c_2).
	$$
	Moreover, since $\bb_{\Gamma^*}(c_1c_2) \cong \bb_{\Gamma^*}(P_c(c_1c_2))$ we have that $\col(\Gamma)$ is indeed a crystal 2-polygraph. The other properties of $\col(\Gamma)$ are a consequence of Theorem \ref{thm:col pres is convergent}.
\end{proof}
\subsection{Rewriting with crystal polygraphs}
Let $X$ be a crystal 2-polygraph and recall the associated sets $\text{Step}(X)$, $\text{Seq}(X)$, $\bra(X)$, $\lbr(X)$, and $\text{Crit}(X)$. Here we show that all these sets inherit the crystal structure from $X_1^*$ via the source maps by explicitly stating the values of the structure maps $\text{wt},\ e_i,\ f_i$. Since the definitions for $e_i$ and $f_i$ are very similar, in the rest of this article we shall use the symbol $k$ to denote an element of the set $\{e_i, f_i\ |\ i\in I\}$.

\subsubsection{Crystal structure on rewriting sequences}
Consider $\text{Step}(X)$, and an element $\gamma=t\alpha v:tuv\Longrightarrow tu'v\in \text{Step}(X)$ for $t,v\in X_1^*$ and $\alpha\in X_2$. Since the congruence $\sim_{X_2}$ is a crystal congruence, we have
$
\bb_{X_1^*}(tuv)\cong \bb_{X_1^*}(tu'v).
$
thus we equip $\text{Step}(X)$ with a crystal structure by defining the structure maps as follows: the weight map is given by $\text{wt}(\gamma) = \text{wt}(tuv)$, and the Kashiwara operators are given by
$$
k.\gamma = \left\{\begin{array}{ll} (k.t)\alpha v,& \quad\text{ if } k.(tuv)=(k.t)uv \\ t(k.\alpha)v,& \quad\text{ if } k.(tuv)=t(e_i.u)v \\ t\alpha(k.v),& \quad\text{ if } k.(tuv)=tu(k.v) \\ 0, & \quad\text{ if } k.(tuv)= 0   \end{array}\right.
$$

Consider $\text{Seq}(X)$ and a rewriting sequence
$$
\mathfrak{s}: w_0 \stackrel{\alpha_1}{\Longrightarrow} w_1\stackrel{\alpha_2}{\Longrightarrow} \cdots \stackrel{\alpha_n}{\Longrightarrow} w_{n}\stackrel{\alpha_{n+1}}{\Longrightarrow}\cdots.
$$
Since $\alpha_i$ are rewriting steps, from the crystal structure of $\text{Step}(X)$ we have that
$
\bb_{X_1^*}(w_0) \cong \bb_{X_1^*}(w_k)
$
for all the words $w_k$ appearing in $\mathfrak{s}$. Thus we equip $\text{Seq}(X)$ with a crystal structure by defining the structure maps as follows: the weight map is given by $\text{wt}(\mathfrak{s})=\text{wt}(w_0)$ and the Kashiwara operator $k$ is defined by setting $k.\mathfrak{s} = 0$ if $k.w_0=0$, and otherwise if $k.w_0\neq 0$, we set
$$
\begin{array}{c}
k.\mathfrak{s}: k.w_0 \stackrel{k.\alpha_1}{\Longrightarrow} k.w_1\stackrel{k.\alpha_2}{\Longrightarrow} \cdots \stackrel{k.\alpha_n}{\Longrightarrow} k.w_{n}\stackrel{k.\alpha_{n+1}}{\Longrightarrow}\cdots
\end{array}
$$
The following result, which is a direct consequence of the definition of the crystal structure on $\text{Seq}(X)$ shows the interaction of the crystal structure with the property of termination.
\begin{prop}\label{prop:some properties of K2} Let $X$ be a crystal 2-polygraph, $\mathfrak{s}_1,\mathfrak{s}_2\in \text{Seq}(X)$, $k$ a Kashiwara operator such that $k.\mathfrak{s}_1\neq 0$ and $k.\mathfrak{s}_2\neq 0$, and $m,n\in\mathbb{N}$. Then
	\begin{itemize}
		\item[{\bf i)}] $k.\mathfrak{s}_1(n)=k.(\mathfrak{s}_1(n))$, where $\mathfrak{s}_1(n)$ denotes the $n$-th entry of $\mathfrak{s}_1$,
		\item[{\bf ii)}] $k.\mathfrak{s}_1$ is finite if and only if $\mathfrak{s}_1$ is finite, and in that case $\text{len}(\mathfrak{s}_1)=\text{len}(k.\mathfrak{s}_1)$,
		\item[{\bf iii)}] $\mathfrak{s}_1(m)=\mathfrak{s}_2(n)$ if and only if $(k.\mathfrak{s}_1)(m)=(k.\mathfrak{s}_2)(n)$.
	\end{itemize}
\end{prop}

\subsubsection{Crystal structure on branchings}
Consider $\text{Br}(X)$ and $(\alpha,\beta)\in\bra$ where $\alpha:u\Longrightarrow v$ and $\beta:u \Longrightarrow w$. Since $\bb_{\step(X)}(\alpha)\cong\bb_{\step(\beta)}(\beta)\cong\bb_{X_1^*}(u)$, we define the crystal structure maps as follows: the weight map is defined as $text{wt}((\alpha,\beta))=\text{wt}(u)$, and the Kashiwara operator $k$ is defined by setting $k.(\alpha,\beta)=0$ if $k.u = 0$, and otherwise if $k.u\neq 0$ we set $k.(\alpha,\beta) = (k.\alpha,k.\beta)$. Graphically for $k=e_i$ we interpret this as 
$$
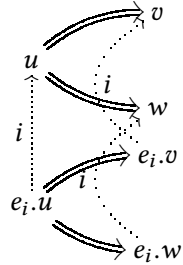
\xymatrix@!R@R=.5em{ & v & \\ u \ar@{=>}@/^/ [ru] \ar@{=>}@/_/ [rd]  & &\\ & w & \\ & {e_i.v} \ar@{.>}@/^5ex/@[kuq] [uuu] _i & \\ {e_i.u} \ar@{.>}@[kuq] [uuu] ^i \ar@2@/^/@[black] [ru] \ar@2@/_/@[black] [rd] & & \\ & {e_i.w} \ar@{.>}@/^5ex/@[kuq] [uuu] ^i &
}
$$
The next result shows that the particular types of branchings are preserved by the crystal structure.
\begin{prop}\label{prop:kash on branchings}
	Let $X$ be a crystal 2-polygraph, $(\alpha,\beta)\in\text{Br}(X)$, and $k$ a Kashiwara operator such that $k.(\alpha,\beta)\neq 0$. Then
	\begin{itemize}
		\item[{\bf i)}] $(\alpha,\beta)\in \text{LBr}(X)$ if and only if $k.(\alpha,\beta)\in\text{LBr}(X)$;
		\item[{\bf ii)}] $(\alpha,\beta)$ is confluent if and only if $k.(\alpha,\beta)$ is confluent,
		\item[{\bf iii)}] $(\alpha,\beta)$ is aspherical (resp. Peiffer) if and only if $k.(\alpha,\beta)$ is aspherical (resp. Peiffer). In particular $(\alpha,\beta)$ is overlapping if and only if $k.(\alpha,\beta)$ is overlapping
		\item[{\bf iv)}] $(\alpha,\beta) \in\text{Crit}(X)$ if and only if $k.(\alpha,\beta)\in\text{Crit}(X)$.
	\end{itemize}
\end{prop}

\begin{proof} 
	
	The statements \textbf{i)}, \textbf{ii)}, and \textbf{iii)} follow directly from 
	the definitions of different families of branchings. Here we prove \textbf{iv)}.
	
	Recall from \eqref{eq:types of crits} that there are two types of critical branchings: inclusion, and overlapping. Here we shall prove part \textbf{iv)} for the overlapping branchings, as the proof for the inclusion branchings is entirely similar. Consider the case of $k=e_i$. Let $(\alpha,\beta)$ be an overlapping critical branching of $X$ so that $e_i.(\alpha,\beta)\neq 0$. By {\bf i)} we see that $e_i.(\alpha,\beta)$ is a 
	local branching, and by {\bf iii)} it is overlapping. We 
	need to show that $(e_i.\alpha,e_i.\beta)$ is minimal with respect to the order $\sqsubset$ as defined in \ref{subsub:critical branchings}. Suppose 
	the contrary, that $(e_i.f,e_i.g)$ is not minimal and thus there exists a branching $(\gamma,\delta)\sqsubset 
	(e_i.\alpha,e_i.\beta)$. Without loss of generality we may assume that $(\delta,\gamma)$ is itself minimal with respect to the 
	order $\sqsubset$, thus it is a critical branching. 
	
	Since $(\gamma,\delta)\sqsubset (e_i.\alpha,e_i.\beta)$ we have that
	$$
	e_i.\alpha = p\gamma q\ \ \text{and }\ e_i.g=p \delta q,
	$$
	for some $p,q\in X_1^*$. By applying the $f_i$ Kashiwara operator to the branching $(e_i.\alpha,e_i.\beta)$, we get
	$$
	\alpha = f_i.(p\gamma q)\ \ \text{and }\ g=f_i.(p\delta q).
	$$
	By definition of the Kashiwara operators $f_i$, we have that the action of $f_i$ on the $2$-cells $p\gamma q$ and $p\delta q$ is entirely determined by the action of $f_i$ on their sources, which coincide and are equal to $p u q$ with $u=s_1(\gamma)=s_1(\delta)$ We consider now the three possible cases for the action of $f_i$ on the word $puq$.
	\begin{itemize}
		\item[\di] $f_i.(puq)=(f_i.p)uq$. We then have that $f_i.(p\gamma q)=(f_i.p)\gamma q$ and $f_i.(p\delta q) = (f_i.p)\delta q$, and thus
		$$
		(\alpha,\beta) = ((f_i.p)\gamma q, (f_i.p)\delta q) \sqsupset (\gamma,\delta)
		$$
		implying that $(\alpha,\beta)$ is not critical, thus a contradiction.
		\item[\di] $f_i.(puq)= p (f_i.u)q$. We then have that $f_i.(p\gamma q) = p(f_i.\gamma) q$ and $f_i.(p\delta q) = p(f_i.\delta) q$, and thus
		$$
		(\alpha,\beta) = (p(f_i.\gamma)q,p(f_i.\delta)q) \sqsupset f_i.(\gamma,\delta)
		$$
		implying that $(\alpha,\beta)$ is not critical, thus a contradiction.
		\item[\di] $f_i.(puq)=pu(f_i.q)$. We then have that $f_i.(p\gamma q)=p\gamma (f_i.q)$ and $f_i.(p\delta q) = p\delta (f_i.q)$, and thus
		$$
		(\alpha,\beta) = (p\gamma (f_i.q), p\delta (f_i.q)) \sqsupset (\gamma,\delta)
		$$
		implying that $(\alpha,\beta)$ is not critical, thus a contradiction.
	\end{itemize}
	The assumption that $e_i.(\alpha,\beta)$ is not a critical branching leads to a contradiction in each of these cases, hence we conclude that $e_i.(\alpha,\beta)$ is indeed a critical branching.
	
	The case of $k=f_i$ is proven entirely analogously.
\end{proof}

\subsection{Rewriting on highest weights} \label{sec:Newman and CPL crystal}

Let $\Gamma$ be a crystal. We denote by $\Gamma^0$ the set of elements of highest weight in $\Gamma$. Since $\Gamma$ is an object of $\mathfrak{g}-\textsc{cryst}$, each of the connected components of $\bb_{\Gamma^*}(w)$ contains a unique word of highest weight, that is $\bb_{\Gamma^*}(w)^0=\{w^0\}$ for some $\lambda\in P^+$.

\begin{prop}
	Let $M$ be a crystal monoid. Then $M^0$ is a submonoid of $M$.
\end{prop}
\begin{proof}
Let $m_1,m_2\in M^0$. From Theorem \ref{thm:char of crystal monoids} we have $\bb_M(m_1m_2)\cong \bb_{M\otimes M}(m_1\otimes m_2)$, which implies that $m_1m_2\in M^0$ if and only if $m_1\otimes m_2 \in (M\otimes M)^0$. Let now $i\in I$. By definition of the tensor product of crystals we have that $e_i.(m_1\otimes m_2)$ equals one of  $(e_i.m_1)\otimes m_2,\ m_1\otimes(e_i.m_2)$ or $0$. But since $m_1$ and $m_2$ are of highest weight, we have $e_i.m_1 = e_i.m_2=0$, implying that $e_i.(m_1\otimes m_2)=0$. This shows that $m_1m_2\in M^0$ hence $M^0$ is indeed a submonoid of $M$.
\end{proof}

\begin{exo}
	Consider $\Gamma = A_n, B_n, C_n$. Then the elements of highest weight in $Pl(\Gamma)$ are the tableaux with $k$-rows, $1\leq k \leq n$ such that all the entries in the $i$-th row are equal to $i$ (see \cite{kashiwara1991crystal},\cite{lecouvey2002schensted}). If we denote by $c_k$ a column with $k$ boxes, whose reading is $12\cdots k$, and since
	$$
	P_c(c_ic_j) = \begin{cases}
	c_ic_j &\ \text{if } i\leq j\\
	c_j c_i &\ \text{otherwise}
	\end{cases}
	$$
	we see that $Pl(\Gamma)^0=\text{Com}(n)$, the commutative monoid on $n$ generators.
\end{exo}
Next we obtain reduced versions of Newman's Lemma and the Critical Pair Lemma for crystal 2-polygraphs.
\subsubsection{Rewriting system at highest weight} Let $X$ be a crystal 2-polygraph. To $X$ we associate the \textit{rewriting system at highest weight}
$$
\mac{R}^0(X)= \left( (X_1^*)^0 \ |\  \step(X)^0 \right).
$$
This is an abstract rewriting system, and we have notions of rewriting sequences, branchings, and local branchings at highest weight. Moreover, the notion of critical branchings is carried over to $\mac{R}^0(X)$ via $\text{Crit}(\mac{R}^0):=\text{Crit}(X)\cap \text{Br}(\mac{R}^0)$. Then from Propositions \ref{prop:some properties of K2} and \ref{prop:kash on branchings} we obtain the following

\begin{prop}
	Let $X$ be a crystal 2-polygraph, and let $\mathcal{P}$ denote one of the properties of 2-polygraphs: termination, confluence, local confluence, confluence of critical branchings. Then $X$ has the property $\mathcal{P}$ if and only if $\mac{R}^0(X)$ has the property $\mac{P}$.
\end{prop}
This result gives reduced versions of Newman's Lemma and the Critical Pair Lemma for crystal 2-polygraphs.
\begin{cor} Let $X$ be a crystal 2-polygraph, and $\mac{R}^0(X)$ the rewriting system at highest weight.
	\begin{itemize}
		\item[\textbf{i)}] \textbf{(Newmann's Lemma)} If $\mac{R}^0(X)$ is terminating and locally confluent, then $X$ is terminating and confluent.
		\item[\textbf{ii)}] \textbf{(Critical Pair Lemma)}  If the critical branchings of $\mac{R}^0(X)$ are confluent, then $X$ is locally confluent.
	\end{itemize}
\end{cor}
%

\section{Crystal structure on free 2-monoids} \label{sec:crystal structure on squier}
In Section \ref{sec:K-rewriting} we have introduced the class of crystal monoids and the notion of a crystal 2-polygraph as the adapted polygraph for studying crystal monoids via rewriting theory. We have seen that the crystal structure extends to the sets of rewriting sequences and to the different kinds of branchings, which leads to reduced verifications of certain rewriting properties of crystal polygraphs.

In this section we show that the crystal structure of a 2-polygraph $X$ extends to the free 2-monoids $X_2^*$ and $X_2^\top$, as described in \ref{subsec:2 cats and 2 dogs}. The approach towards doing this consists in showing that the set of $i$-cells admits a crystal structure, and we specify the compatibility of the Kashiwara operators with the $\star_0$ and $\star_1$ compositions. In particular we show that for a convergent presentation $X$, Squier's coherent extension $\Omega$ also admits a crystal structure, which further extends to the free (3,1)-monoid generated by it. This in turn reduces the computation of coherence from convergence to considerations at highest weight. In what follows, we denote the Kashiwara operators $\{e_i,f_i\ |\ i\in I\}$ simply by $k$.

\subsection{2-dimensional crystal monoids}

We introduce here a notion of a 2-dimensional crystal monoid. This is simply a 2-monoid whose underlying sets of 1-cells and 2-cells are crystals, and there is compatibility between the Kashiwara operators and the categorical compositions $\star_0$ and $\star_1$.

\begin{defin}\label{def:crystal 2 monoids} Let $m\leq 2$. A crystal $(2,m)$-monoid is a $(2,m)$-monoid $\mac{C}$ that admits a crystal structure as follows:
	\begin{itemize}
		\item[{\bf (CM1)}] the set of $1$-cells of $\mac{C}$ is a crystal monoid with respect to $\star_0$,
		\item[{\bf (CM2)}] the set of $2$-cells of $\mac{C}$ is a crystal,
		\item[{\bf (CM3)}] the source and target maps from the $2$-cells of $\mac{C}$ to the $1$-cells in $\mac{C}$ are crystal morphisms,
		\item[{\bf (CM4)}] for $\alpha,\beta$ $2$-cells in $\mac{C}$ and a Kashiwara operator $k$ we have
		$$
		k.(\alpha \star_0 \beta)=\left\{\begin{array}{ll} (k.\alpha)\star_0 \beta, &\quad \text{if } k.(s(\alpha)s(\beta))=(k.s(\alpha))s(\beta) \\ \alpha\star_0(k.\beta), &\quad \text{if } k.(s(\alpha)s(\beta)) = s(\alpha)(k.s(\beta))\\ 0, &\quad \text{otherwise.}\end{array}\right.
		$$
		$$
		k.(\alpha\star_1\beta)=\left\{\begin{array}{ll} (k.\alpha)\star_1(k.\beta), &\quad \text{if } k.(s(\alpha))\neq 0 \\ 0, & \quad \text{otherwise}. \end{array}\right.
		$$
	\end{itemize}
\end{defin}

We remark here that as there is a single $0$-cell in a crystal $(2,m)$-monoid, any two $1$-cells are $\star_0$-composable. Hence often times we denote the composition $u\star_0 v$ by a concatenation $uv$. The idea behind this notion of 2-dimensional crystal monoids is that the free 2-monoid $X_2^*$ and the free $(2,1)$-monoid $X_2^\top$ generated by $X$ both satisfy the conditions of Definition \ref{def:crystal 2 monoids}.

\subsubsection{Crystal structure on $X_2^*$}
Let $X=(X_1,X_2)$ be a crystal $2$-polygraph, and consider the free $2$-monoid $X_2^*$ generated by $X$ as in \ref{subsec:free2monoids}. By definition, there is a crystal structure on the 1-cells $X_1^*$ and on the generating $2$-cells $X_2$. Here we extend the crystal structure to all the $2$-cells of $X_2^*$ so that it is compatible with Definition \ref{def:crystal 2 monoids}. We do this progressively: first for identity 2-cells and the generating 2-cells of $X_2$; and then for the $\star_0$ and $\star_1$ compositions of such cells.

Consider a 1-cell $u\in X_2^*$ and a 2-cell $\gamma:u\Longrightarrow v\in X_2^*$. We define the weight map $\text{wt}(\gamma)=\text{wt}(s_1(\gamma))$, and the Kashiwara operator $k$ first on the identity 2-cell $1_u$ and then on $\gamma$ via
\begin{itemize}
	\item[\di] $k.1_u=1_{k.u}$ if $k.u\neq 0$, and otherwise $k.1_u = 0$, $u\in X_2^*$
	
	\item[\di] $k.\gamma : k.u \Longrightarrow k.v$ if $k.u \neq 0$, and $k.\gamma = 0$ otherwise.
\end{itemize}

We then extend the crystal structure to $\star_0$ and $\star_1$ compositions of 2-cells by defining the weight map as $\text{wt}(\alpha\star_0\beta) = \text{wt}(\alpha) + \text{wt}(\beta)$, and $\text{wt}(\alpha\star_1\beta) = \text{wt}(\alpha) = \text{wt}(\beta)$, and by defining the Kashiwara operators as in \textbf{(CM4)}.

We show that the definition of the Kashiwara operators on $2$-cells of $X_2^*$ is compatible with the defining axioms of $X_2^*$. We thus verify the axioms \eqref{eq:peiffer branchings} and \eqref{eq:ugly}. 

\begin{itemize}
	\item[$(FM1)$] $\alpha w v \star_1 u' w \beta \equiv uw\beta \star_1 \alpha w v'$, for  $\alpha:u\Longrightarrow u'$, $\beta:v\Longrightarrow v'$ in $X_2^*$, and $w\in X_1^*$.
	
	\noindent Consider a Kashiwara operator $k$ such that  $k.(\alpha w v \star_1 u' w \beta) \neq 0$. Then by the action of $k$ defined as in \textbf{(CM4)} we have that $k.s_1(\alpha w v \star_1 u' w \beta)\neq 0$, and from \eqref{eq:peiffer branchings} we have
	$$
	s_1(\alpha w v \star_1 u' w \beta) = s_1(uw\beta \star_1 \alpha w v'),
	$$ 
	hence $k.s_1(uw\beta \star_1 \alpha w v')\neq 0$, and thus $k.(uw\beta\star_1 \alpha w v')\neq 0$. Then by \textbf{(CM4)} we obtain
	\begin{equation}\label{eq:part for A1}
	\begin{array}{c}
	k.(\alpha w v \star_1 u' w \beta)= (k.(\alpha w v ))\star_1 (k.(u' w \beta))\\
	k.(uw\beta \star_1 \alpha w v')=(k.(uw\beta))\star_1(k.(\alpha w v')).
	\end{array}
	\end{equation}
	The operator $k$ acts on $\alpha w v'$ by acting on precisely one of the \textit{components} $\alpha$, $w$, or $v'$. We distinguish these three cases as follows. Say if $k$ acts on $\alpha w v$ by acting on its $l$-th component for $l\in\{1,2,3\}$, then $k$ acts on the $u'w\beta$, $uw\beta$, and $\alpha w v'$ by acting on their $l$-th components as well. Thus from \eqref{eq:part for A1} we have
	$$
	\begin{array}{lc}
	l = 1:\ & k.(\alpha w v \star_1 u'w\beta) = (k.\alpha)wv\star_1 (k.u')w\beta \\
	& k.(uw\beta\star_1 \alpha w v') = (k.u)w\beta\star_1 (k.\alpha)wv'\\
	$\ $&$\ $  \\
	l = 2:\ & k.(\alpha w v \star_1 u'w\beta) = \alpha (k.w)v\star_1 u'(k.w)\beta \\
	& k.(uw\beta\star_1 \alpha w v') =u'(k.w)\beta\star_1 \alpha(k.w)v'\\
	$\ $ & $\ $ \\
	l = 3:\ & e_i.(\alpha w v \star_1 u'w\beta) = \alpha w e_i.v\star_1 u'w(e_i.\beta) \\
	& k.(uw\beta\star_1 \alpha w v') =u'w(k.\beta)\star_1 \alpha w(k.v')	 
	\end{array}
	$$
	and we see that in each of these cases we indeed have
	$$
	k.(\alpha w v \star_1 u'w\beta) \equiv k.(uw\beta\star_1 \alpha w v').
	$$
	Thus the definition of the Kashiwara operators $e_i$ and $f_i$ on $X_2^*$ is compatible with \eqref{eq:peiffer branchings}.
	
	\item[$(FM2)$]
	\begin{equation}\label{eq:A2}
	\begin{array}{c}
	(u_1\alpha_1u_1'\star_1 \cdots \star_1 u_m\alpha_m u'_m)\star_0 (v_1\beta_1 v_1' \star_1\cdots \star_1 v_n\beta_n v_n')\\
	= u_1\alpha_1 u_1'v_1s(\beta_1)v_1'\star_1\cdots \star_1 u_m\alpha_m u_m'v_1s(\beta_1)v_1'\\
	\star_1 u_m t(\alpha_m)u_m'v_1\beta_1 v_1' \star _1 \cdots \star_1 u_m t(\alpha_m)u_m'v_n\beta_nv_n'.
	\end{array}
	\end{equation} 
	Denote by $L$ respectively $R$ the expression to the left respectively to the right of $\star_0$ in the first line of \eqref{eq:A2}, and by $F$ the expression in the second and third lines of \eqref{eq:A2}. In other words, we rewrite \eqref{eq:A2} as
	$
	L\star_0 R = F.
	$
	Let $k$ be a Kashiwara operator such that $k.(L\star_0 R)\neq 0$. By \textbf{(CM4)} we have $k.(L\star_0 R)= (k.L)\star_0 R$ or $k.(L\star_0 R)= L \star_0(k.R)$. We consider here the first case, as the second one is handled similarly.
	
	Since $k.(L\star_0 R)=(k.L)\star_0 R$, we have that $k.(s(L\star_0 R))=k.(s(L)s(R))=(k.s(L))s(R)$. This implies that for any $1\leq p \leq m$ and $1\leq q \leq n$ we have
	\begin{equation}\label{eq:proof axiom 2}
	k.(u_p \alpha_p u'_k v_q s(\beta_q)v_q')= k.(u_p \alpha_p u'_p)v_q s(\beta_q)v_q'
	\end{equation}
	
	Then by \textbf{(CM4)} and \eqref{eq:ugly} we obtain
	$$
	\begin{array}{c}
	k.(L\star_0 R) = (k.L)\star_0 R = (k.(u_1\alpha_1u_1')\star_1 \cdots \star_1 (k.(u_m\alpha_m u_m'))) \star_0 R=\\
	= (k.(u_1\alpha_1 u_1'))v_1s(\beta_1)v_1'\star_1\cdots \star_1 (k.(u_m\alpha_m u_m'))v_1s(\beta_1)v_1'\\
	\star_1 (k.(u_mt(\alpha_m)u_m'))v_1\beta_1 v_1' \star _1 \cdots \star_1 (k.(u_m t(\alpha_m)u_m'))v_n\beta_nv_n'
	\end{array}
	$$
	and finally from \eqref{eq:proof axiom 2} we obtain
	$$
	\begin{array}{c}
	k.(L\star_0 R) = (k.(u_1\alpha_1 u_1'v_1s(\beta_1)v_1'))\star_1\cdots \star_1 (k.(u_m\alpha_m u_m'v_1s(\beta_1)v_1'))\\
	\star_1 (k.(u_m t(\alpha_m)u_m'v_1\beta_1 v_1')) \star _1 \cdots \star_1 (k.(u_m t(\alpha_m)u_m'v_n\beta_nv_n')) = k.F.
	\end{array}
	$$
\end{itemize}
This shows that the action of the Kashiwara operators on $2$-cells is compatible with the defining axioms $(FM1)$ and $(FM2)$ of the free 2-monoid $X_2^*$ generated by $X$ as in \ref{subsec:2 cats and 2 dogs}.

Next we show how the crystal structure is extended to free (2,1)-monoids generated by crystal polygraphs.

\subsubsection{Crystal structure on $X_2^\top$} Let $X=(X_1,X_2)$ be a crystal 2-polygraph. Recall the $2$-polygraph $X^{\pm}=(X_1,X_2\sqcup X_2^-)$, as in \ref{subsec:free2monoids}. We define a crystal structure on $X^-$ by setting $\text{wt}(\alpha^-) = \text{wt}(\alpha)$, and the Kashiwara operators $k$ as follows: for $\alpha\in X_2$ such that $k.\alpha \neq 0$, we set $k.\alpha^- = (k.\alpha)^-$. Then $X^\pm$ is a crystal 2-polygraph and we can define a crystal structure on the set of 2-cells of $(X_2^\pm)^*$ as in the previous paragraph. The following proposition shows that the Kashiwara operators descend to the 2-cells of $X_2^\top$.

\begin{prop} \label{prop: jsut a random proposition} Let $X$ be a crystal 2-polygraph, $\alpha\in X_2$, $u,v\in X_1^*$, and $k$ a Kashiwara operator such that $k.(u\alpha v)\neq 0$. Then
	$$
	k.(u\alpha v\star_1 u\alpha^-v)\equiv1_{k.(us(\alpha)v)}\quad \text{and}\quad k.(u\alpha^-v\star_1 u \alpha v)\equiv1_{k.(ut(\alpha)v)}\\
	$$
\end{prop}
\begin{proof}
	We prove here the first relation, as the other one is done analogously. We consider the three possibilities of the action of $e_i$ on $u\alpha v$.
	
	If $k.(u\alpha v)=(k.u)\alpha v$, we have that $k.(us(\alpha)v)=(k.u)s(\alpha)v$. This means that $k.(ut(\alpha^-)v)=(k.u)t(\alpha^-)v$ hence we get
	$$
	k.(u\alpha v\star_1 u\alpha^-v)=(k.u)\alpha v \star_1 (k.u)\alpha^-v \equiv1_{(k.u)s(\alpha)v)}=1_{k.(us(\alpha)v)}.
	$$
	Similarly we prove the case of $k$ acting on $v$.
	
	If $k.(u\alpha v)= u (k.\alpha) v$. We then have $k.(u\alpha^- v) = u (k.\alpha^-) v= u(k.\alpha)^-v$, thus indeed we get
	$$
	k.(u\alpha v\star_1 u\alpha^-v)\equiv1_{k.(us(\alpha)v)}
	$$
	which is what we wanted to show.
\end{proof}

Thus $X_2^\top$ also admits a crystal structure, and by the constructions it is evident that both $X_2^*$ and $X_2^\top$ satisfy the conditions of Definition \ref{def:crystal 2 monoids}. We summarize this in the following result
\begin{thm}\label{thm:free crystal 2-mon}
	Let $X$ be a crystal 2-polygraph. Then the free 2-monoid $X_2^*$, respectively the free (2,1)-monoid $X_2^\top$ generated by $X$ is a crystal 2-monoid, respectively crystal (2,1)-monoid.
\end{thm}

Similarly to the discussion in \ref{sec:Newman and CPL crystal}, we specify here the highest-weight subobjects of crystal 2-monoids.

\begin{defin}
	Let $\mac{C}$ be a crystal 2-monoid. The full subcategory of $\mac{C}$ consisting of the 1-cells and 2-cells of highest weight, is a 2-submonoid of $\mac{C}$ and we denote it by $\mac{C}^0$.
\end{defin}
\subsection{Coherent presentations for crystal monoids}

Here we adapt the notions of congruences on 2-monoids, cellular extensions, and normalization strategies to the context of crystal monoids. Further we prove the main result of this article which states that Squier's coherent extension of a convergent crystal 2-polygraph $X$ is entirely determined by the parallel cells of highest weight of $X_2^\top$.

\subsubsection{Crystal congruences on crystal 2-monoids}\label{subsec:crystal on spheres} Let $C$ be a crystal 2-monoid, and recall the set of 2-spheres $\text{Sph}(\mac{C})$ of $\mac{C}$. The elements of $\text{Sph}(\mac{C})$ are pairs of 2-cells $(f,g)$ that have the same source and same target. This implies that the crystal structure of $\mac{C}$ extend to $\text{Sph}(\mac{C})$. Indeed, we the weight map on $\sph(\mac{C})$ is given by $\text{wt}((f,g)) = \text{wt}(s_1(f))$, and the Kashiwara operator $k$ acts via $k.(f,g) = (k.f,k.g)$ if $k.s_1(f)\neq 0$. We now define the following concepts in the context of crystal monoids.

\begin{defin}
	Let $\mac{C}$ be a crystal 2-monoid, $\equiv$ a congruence on $\mac{C}$, $\Omega$ a cellular extension of $\mac{C}$, and $X=(X_1,X_2,X_3)$ a (3,1)-polygraph
	\begin{itemize}
		\item[\textbf{i)}] $\equiv$ is called a \textit{crystal congruence} on $\mac{C}$ if for any Kashiwara operator $k$ and relation between 2-cells $f \equiv g$ such that $k.f\neq 0$, we have $k.f \equiv k.g$.
		\item[\textbf{ii)}] $\Omega$ is called a \textit{crystal cellular extension} if $\Omega$ is a subcrystal of $\sph(\mac{C})$.
		\item[\textbf{iii)}] $X$ is called a \textit{crystal (3,1)-polygraph} if $(X_1,X_2)$ is a crystal 2-polygraph, and $X_3$ is a crystal cellular extension of $X_2^\top$.
	\end{itemize}
\end{defin}
Note that for a crystal cellular extension $\Omega$ of $\mac{C}$, the congruence $\equiv_\Omega$ on $\mac{C}$ generated by $\Omega$ is a crystal congruence. Recall from Theorem \ref{thm:Squier} that Squier's coherence theorem extends a convergent presentation to a coherent one by taking $\Omega$ to be a family of generating confluences. The notion of normalization strategies as in \ref{subsec:norm} gives an algorithmic method of determining a family of generating confluences. Next we consider a type of normalization strategy compatible with the crystal structure.

\subsubsection{Normalization strategies for crystal 2-polygraphs}\label{subsec:normalization strategies}
Let $X$ be a crystal 2-polygraph and $\overline{X}$ the crystal monoid it presents, and a section $r:\ov{X}\longrightarrow X_1^*$ of $\ov{X}$. Recall that we often denote the image $r(w)$ by $\widehat{w}$. We say that $r$ is a \textit{crystal section} if $r$ is a crystal morphism. In particular, for a crystal section $r$, $w\in X_1^*$, and a Kashiwara operator $k$ such that $k.w\neq 0$, we have $\widehat{k.w} = k.\widehat{w}$. Note that if $X$ is convergent, one can take $r$ to be such that $\widehat{w}$ is the normal form of $w$. Clearly then $r$ is a crystal section.

Consider now a normalization strategy $\sigma: X_1^*\longrightarrow X_2^\top$ for a crystal section $r$, that is a map from $X_1^*$ to the set of $2$-cells in $X_2^\top$. Let $w\in X_1^*$ and $k$ a Kashiwara operator such that $k.w\neq 0$. Then we have
\begin{equation}\label{eq:norm}
\sigma(k.w)= (k.w\Longrightarrow \widehat{k.w})=k. (w\Longrightarrow \widehat{w}) = k.\sigma(w).
\end{equation}
We call such a $\sigma$ a \textit{crystal normalization strategy.} We then have the following version of Squier's coherent completion theorem in the context of crystal monoids.

\begin{thm}[Squier's coherence theorem for crystal 2-polygraphs]\label{thm:squier 2.0}
	Let $X$ be a convergent crystal 2-polygraph, $r$ a crystal section, and $\sigma$ a crystal normalization strategy for $r$. Let $\Omega$ be the family of generating confluences obtained from $r$ and $\sigma$. Then $\Omega$ is a crystal cellular extension of $X_2^\top$, and the 2-spheres in $\Omega$ are entirely determined by the critical branchings of highest weight of $X$.
\end{thm}
\begin{proof}
	By definition of $\Omega$ in terms of $r$ and $\sigma$ as in \ref{subsec:norm}, $\Omega$ consists of spheres of $X_2^\top$ of the form $\gamma = (f \star_1 \sigma_{t_1(f)}, g \star_1 \sigma_{t_1(g)})$ for the critical branchings $(f,g)$ of $X$. Let now $k$ be a Kashiwara operator such that $k.f \neq 0$. We compute the action of $k$ on the sphere $\gamma$ as follows. First, by the crystal structure on $\text{Sph}(X)$ as in \ref{subsec:crystal on spheres} we have
	$$
	k.\gamma = k.\left(f\star_1 \sigma_{t_1(f)}, g\star_1 \sigma_{t_1(g)}\right)= \left(k.(f\star_1 \sigma_{t_1(f)}), k.(g\star_1 \sigma_{t_1(g)})\right).
	$$
	By Theorem \ref{thm:free crystal 2-mon} we have that $X_2^\top$ is a crystal 2-monoid, and by Definition \ref{def:crystal 2 monoids} \textbf{(CM4)} we obtain
	$$
	k.\gamma = \left((k.f)\star_1 (k.\sigma_{t_1(f)}), (k.g)\star_1 (k.\sigma_{t_1(g)})\right).
	$$
	Then from \eqref{eq:norm} we obtain $k.\sigma_{t_1(f)} = \sigma_{k.t_1(f)}$ and $k.\sigma_{t_1(g)}= \sigma_{k.t_1(g)}$. Now since the target map $t_1$ is a crystal morphism by Definition \ref{def:crystal 2 poly}, we have $k.t_1(f) = t_1(k.f)$ and $k.t_1(g)= t_1(k.g)$. Thus we finally obtain
	$$
	k.\gamma = \left((k.f)\star_1 \sigma_{t_1(k.f)}, (k.g)\star_1 \sigma_{t_1(k.g)}\right)
	$$
	By Proposition \ref{prop:kash on branchings} we have that $(k.f,k.g)$ is a critical branching, hence the 2-sphere $k.\gamma$ is the confluence diagram of $(k.f,k.g)\in\text{Crit}(X)$ determined by the crystal section $r$ and the crystal normalization strategy $\sigma$. This proves the first part of the theorem, i.e. that $\Omega$ is a subcrystal of $\text{Sph}(X)$.
	
	For the second part, consider the set of critical branchings at highest weight $\Lambda := \Omega^0$. Then via $r$ and $\sigma$ we can compute the highest weight elements of the family of generating confluences $\Omega$, which we deonte by $\Omega^0$. Consider now any critical branching $(f,g)\in \text{Crit}(X)$, and its highest weight $(f,g)^0$. Let $\gamma^0$ be the 2-sphere in $\Lambda$ corresponding to the critical branching $(f,g)^0$. Since $(f,g)^0$ is the unique highest weight vertex in the connected component $\bb_{\text{Crit}(X)}(f,g)$, there exists a sequence $k_1,k_2,\dots,k_l$ of $f_i$ Kashiwara operators (that is $k_j = f_{i_j})$ for some $i_j\in I$ such that
	$$
	k_1.k_2.\dots.k_l.(f,g)^0=(f,g).
	$$
	Then 2-sphere in $\Omega$ corresponding to $(f,g)$ is simply $k_1.k_2.\dots.k_l.\gamma^0$. This completes the proof of the theorem.
\end{proof}

We give here a graphical interpretation of the second part of the proof of Theorem \ref{thm:squier 2.0}. Let $X$, $r$, and $\sigma$ be as in the statement of the theorem, and $\Lambda$ as in the proof. Let $(f,g)\in\text{Crit}(X)$ and $(f,g)^0$ the corresponding critical branching of highest weight. Then the 2-sphere $\gamma\in\Omega$ corresponding to $(f,g)$ is determined by acting with the adequate Kashiwara operators on the 2-sphere $\gamma^0$ corresponding to $(f,g)^0$ in $\Lambda$. Graphically we can present this as
$$
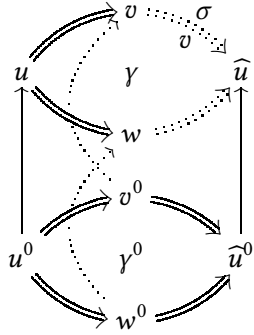
\xymatrix@!R@R=.5em{ & v \ar@{:>}@/^/ ^\sigma_v [rd] & \\ u \ar@{=>}@/^/ [ru] \ar@{=>}@/_/ [rd]  & \gamma & \widehat{u}\\ & w \ar@{:>}@/_/ [ru]& \\ & v^0 \ar@{.>}@/^5ex/@[kuq] [uuu]  \ar@{=>}@/^/ [rd] & \\ {u^0} \ar@1@[kuq] [uuu]  \ar@2@/^/ [ru] \ar@2@/_/ [rd] & \gamma^0 & \widehat{u}^0 \ar@1@[kuq] [uuu] \\ & w^0 \ar@{.>}@/^5ex/@[kuq] [uuu] \ar@{=>}@/_/ [ru] &
}
$$
that is we obtain $\gamma$ via $\gamma^0$ and the Kashiwara operators. Note moreover that $(X_1,X_2,\Omega)$ is a crystal (3,1)-polygraph.
We remark here that a proof of Squier's completion theorem in the general context of monoids (and categories) can be found in the work of Guiraud-Malbos \cite{guiraud2018polygraphs}. In particular they explicitly prove the acyclicity of $\Omega$ by describing how to pave any 2-sphere by composing the elements of $\Omega$.
Guiraud-Malbos \cite{guiraud2018polygraphs}.

Finally, the following result shows that in the case of finite convergent reduced crystal 2-polygraphs, one may choose the leftmost and rightmost normalization strategies as $\sigma$ in Theorem \ref{thm:squier 2.0}.

\begin{prop}\label{prop:left-right}
	Let $X$ be a reduced convergent crystal 2-polygraph, $r$ a crystal section. Then the leftmost and rightmost normalization strategies are crystal normalization strategies.
\end{prop}
\begin{proof}
	Recall the leftmost normalization strategy $\sigma$ as defined in \eqref{eq:leftmost}, given on $u\in X_1^*$ by $\sigma_u = \lambda_u \star_1 \sigma_{t_1(\lambda_u)}$. Let $k$ be a Kashiwara operator such that $k.u\neq 0$.
	
	We first show that $k.\lambda_u = \lambda_{k.u}$. Indeed, consider the sets of rewriting steps with source $u$ and $k.u$ respectively, that is $\text{Step}(X)_u$ and $\text{Step}(X)_{k.u}$. Since $X$ is a crystal 2-polygraph, we have that the mapping $f \longmapsto k.f$ defines a bijection from $\text{Step}(X)_u$ to $\text{Step}(X)_{k.u}$. We show that this bijection is compatible with the orders $\preceq$ on $\text{Step}(X)_u$ and $\step(X)_{k.u}$ as in \ref{subsec:norm}. Consider $f,g\in\step(X)_u$ such that $f\preceq g$. Suppose that $f=t\alpha t_1$ and $g=v\beta v_11$ for some $t,t_1,v,v_1\in X_1^*$ and $\alpha,\beta\in X_2$, such that $|t|<|v|$. Then $k$ acts on $f$ and $g$ by acting on one of their components, i.e. on $t$, $\alpha$, or $t'$ for $f$, and on $v$, $\beta$, or $v'$ for $g$. In each of these cases, $k.f$ and $k.g$ will be of the form $k.f= t'\alpha' t_1'$ and $k.g=v'\beta'v_1'$ with $|t| = |t'|$ and $|v|=|v'|$. Hence we have that $f\preceq g$ implies $k.f \preceq k.g$. Conversely, we consider the map $\text{Step}(X)_{k.u} \longrightarrow \text{Step}(X)_{k^{-1}.(k.u)}$, where $e_i^{-1} : = f_i$ and $f_i^{-1}:= e_i$, and obtain that $k.f \preceq k.g$ implies $f\preceq g$. Since the order $\preceq$ is total, and $\lambda_u$ is the minimal element of $\text{Step}(X)_u$, we have that $k.\lambda_u$ is the minimal element in $\step(X)_{k.u}$, hence $k.\lambda_u = \lambda_{k.u}$.
	
	Next, to show that $\sigma$ is a crystal normalization strategy, we note that by definition of $\sigma$ we have
	$
	\sigma_u = \lambda_u \star_1 \sigma_{u_1},
	$
	where $u \stackrel{\lambda_u}{\Longrightarrow} u_1=t_1(\lambda_u)$. Iterating this, since $X$ is a terminating 2-polygraph, we obtain 
	$$
	\sigma_u = \lambda_u \star_1 \lambda_{u_1} \star_1 \cdots \star_1 \lambda_{u_l}
	$$
	for some rewriting sequence $u\Longrightarrow u_1 \Longrightarrow \cdots \Longrightarrow u_l$ in $\mac{R}(X)$. Then from the first part of the proof and Definition \ref{def:crystal 2 monoids} \textbf{(CM4)} we obtain
	$$
	k.\sigma_u = k.(\lambda_u \star_1 \cdots \star_1 \lambda_{u_l})=(k.\lambda_u)\star_1 \cdots \star_1 (k.\lambda_{u_l})=\lambda_{k.u}\star_1 \cdots \star_1 \lambda_{k.u_l}.
	$$
	On the other hand, since the target map $t_1$ is a crystal morphism, we have $t_1(k.u) =k.t_1(u)=k.u_1$, and we obtain $\sigma_{k.u} = \lambda_{k.u} \star_1 \sigma_{k.u_1}$. Iterating this equation, we obtain
	$$
	\sigma_{k.u} = \lambda_{k.u}\star_1 \cdots \star_1 \lambda_{k.u_l}.
	$$
	which shows that $k.\sigma_u = \sigma_{k.u}$, hence the leftmost normalization strategy is indeed a crystal normalization strategy, which is what we wanted to show.
	
	In a completely analogous fashion one shows that the rightmost normalization strategy is also a crystal normalization strategy.
\end{proof}

\subsection{Coherence for plactic monoids of classical type}

Here we show how one may apply Theorem \ref{thm:squier 2.0} to the column presentations of the plactic monoids of classical type.

Consider the crystal bases of classical type $\Gamma= A_n,B_n,C_n,D_n,G_2$ as in \ref{subsec:classical crystals} and recall the column presentation $\col(\Gamma)$ of the plactic monoids $Pl(\Gamma)$ as in \ref{subsec:col pres}. By Theorem \ref{thm:col pres are good} we have that $\col(\Gamma)$ is a finite convergent crystal 2-polygraph. Hence for a crystal section $r$ and normalization strategy $\sigma$, by Theorem \ref{thm:squier 2.0} we can compute the coherent extension of $\col(\Gamma)$ at highest weight. In order to apply the theorem to $\col(\Gamma)$, we need to specify a crystal section and a crystal normalization strategy. For the crystal section, one can always choose $r$ that sends each word to its corresponding normal form. Since $\col(\Gamma)$ are reduced by Theorem \ref{thm:col pres is convergent}, then by Proposition \ref{prop:left-right} we can choose $\sigma$ to be the leftmost normalization strategy.

In order to apply Theorem \ref{thm:squier 2.0}, we identify the critical branchings of $\col(\Gamma)$. Since the rewriting rules in $\col(\Gamma)$ have sources of length $2$, the critical branchings will be of the form
$$
{\xymatrix@!R@R=1em{ & P_c(c_1c_2)c_3  \\
		c_1c_2c_3 \ar@2@ /^/^{P_cc_3}[ur] \ar@2@/_/ _{c_1P_c}[rd]& \\
		& c_1P_c(c_2c_3) }}
$$
for $c_1,c_2,c_3\in\col(\Gamma)$ satisfying $P_c(c_1c_2)\neq c_1c_2$ and $P_c(c_2c_3)\neq c_2c_3$. Then Theorem \ref{thm:squier 2.0} applied to $\col(\Gamma)$ admits the following form.
\begin{cor}
	Let $\Gamma = A_n, B_n, C_n, D_n, G_2$ be a crystal base of classical type, and $\sigma$ a crystal normalization strategy. Then the coherent extension of $\col(\Gamma)$ is determined by the 2-spheres
\begin{equation}\label{eq:coherence diagrams}
{\xymatrix@!R@R=1em{ & P_c(c_1c_2)c_3 \ar@2@ /^/ [rd] ^{\sigma_{P_c(c_1c_2)c_3}} & \\
		c_1c_2c_3 \ar@2@ /^/^{P_cc_3}[ur] \ar@2@/_/ _{c_1P_c}[rd]& & P_c(c_1c_2c_3)\\
		& c_1P_c(c_2c_3) \ar@2@ /_/ _{\sigma_{c_1P_c(c_2c_3)}}  [ur]}}
\end{equation}
$c_1,c_2,c_3 \in\col(\Gamma)$ such that $c_1c_2c_3$ is of highest weight and $P_c(c_1c_2)\neq c_1c_2$, $P_c(c_2c_3)\neq c_2c_3$.
\end{cor}

We note that for $\Gamma = A_n, B_n, C_n, D_n$, by Proposition \ref{prop:2 col lemmata} the column presentation is \textit{almost quadratic}, meaning that the generating rewriting rules are of the form $c_1c_2 \Longrightarrow w$ with $|w| \leq 2$. For the leftmost normalization strategy $\sigma$, the diagram in \eqref{eq:coherence diagrams} takes the form

\begin{equation}\label{eq:conf diags}
{\xymatrix@!R@R=1em{ & c_1'c_2'c_3 \ar@{=>}[r] ^{c_1'P_c}& c_1'c_2''c_3' \ar@{=>} [r] ^{P_c c_3'} & \cdots \ar@{=>}[r] &w_1 \ar@{=>}[rd] & \\
		c_1c_2c_3 \ar@2@ /^/^{P_c c_3}[ur] \ar@2@/_/ _{c_1P_c}[rd]& & & & & P_c(c_1c_2c_3),\\
		& c_1 d_2d_3 \ar@{=>}[r] ^ {P_c d_3}& d_1d_2'd_3 \ar@{=>}[r]  ^ {d_1 P_c} & \cdots \ar@{=>}[r] & w_2 \ar@{=>}[ru] & }}
\end{equation}
meaning it can be identified via computation of the rewriting rules at highest weight.

In \cite{hage2017knuth} Hage-Malbos identify the shapes of these generating confluence diagrams for $\col(A_n)$. If in the \eqref{eq:conf diags} we denote by $L$ and $R$ respectively the lengths of the top and bottom rewriting sequences, they show that $L, R\leq 3$, showing that the generating coherence relations for $Pl(A_n)$ admit a hexagonal shape. They achieve this result via a case-by-case analysis of the critical branchings and their corresponding confluence diagrams. The constructions presented in this article aim to reduce the identification of coherent extensions to only computations of the rewriting rules at highest weight. The next step towards completing this task naturally requires explicit computation with the elements of highest weight in $(\col(\Gamma)_1^*)^0$. In his PhD thesis the author has introduced combinatorial models in types $A$ and $C$, called Yamanouchi trees, designed to parameterize the words of highest weight in $\col(\Gamma)^*$, and to compute rewriting rules at highest weight. In particular one can use such objects to compute the shapes of the confluence diagrams in \eqref{eq:conf diags}. These combinatorial objects form part of an upcoming preprint by the author.

\printbibliography

\quad

\vfill

%
%
%
%
\end{document}